\documentclass[11pt,a4paper]{article}

\usepackage[a4paper,top=3.8cm,bottom=3.8cm,left=3cm,right=3cm]{geometry}%
\usepackage{hyperref}

\usepackage[T1]{fontenc}
\usepackage{amsmath,amssymb,mathtools,amsthm,amsopn}
 \usepackage{mathpazo}
\usepackage{graphicx}

\usepackage[usenames,dvipsnames]{xcolor}

\usepackage{tocloft}
\newcommand{\mass}{\mu_{z,\z}}

\DeclareMathOperator{\dist}{dist}
\DeclareMathOperator{\diam}{diam}
\renewcommand{\Box}{\operatorname{Box}}

\renewcommand{\Im}{\operatorname{Im}}

\usepackage{fancyhdr}
\usepackage{fourier-orns}
\usepackage{enumitem}
\usepackage{cite}

\usepackage[dayofweek]{datetime}
\usepackage{mathrsfs}
\renewcommand{\epsilon}{\varepsilon}
\renewcommand{\phi}{\varphi}

\newcommand{\step}[1]{\par\medskip\noindent\it#1\rm}

\newcommand{\scalar}[2]{\langle#1,#2\rangle}
\newcommand{\norm}[1]{\left\Vert#1\right\Vert}

\newcommand{\abs}[1]{\lvert#1\rvert}

\newcommand{\z}{\zeta}

\newcommand{\g}{\gamma}

\newcommand{\wh}{\widehat}

\newcommand{\e}{\varepsilon}
\renewcommand{\r}{\varrho}
\renewcommand{\rho}{\varrho}

\newcommand{\s}{\sigma}

\newcommand{\la}{\lambda}

\newcommand{\ol}{\overline}

\renewcommand{\d}{\delta}
\newcommand{\Eucl}{\textup{Euc}}
 
\newcommand{\p}{\partial}

\usepackage{mdwlist}

\newcommand{\R}{\mathbb{R}}

\newcommand{\N}{\mathbb{N}}
\newcommand{\C}{\mathbb{C}}

\renewcommand{\d}{\delta}
\renewcommand{\t}{\tau}

\renewcommand{\a}{\alpha}
\renewcommand{\b}{\beta}

\DeclareMathOperator{\Span}{span}

\numberwithin{equation}{section}

\let\oldbibliography\thebibliography
\renewcommand{\thebibliography}[1]{%
  \oldbibliography{#1}%
  \setlength{\itemsep}{0pt}%
}

\usepackage{titlesec}
\titleformat{\section}{%
\normalfont\large\bfseries}{\thesection.}{1em}{}
\titleformat{\subsection}{%
\normalfont\normalsize\bfseries}{\thesubsection.}{1em}{}

\theoremstyle{plain}
\newtheorem {theorem}{Theorem}[section]
\newtheorem {lemma}[theorem]{Lemma}
\newtheorem {corollary} [theorem]{Corollary}

\newtheorem {proposition} [theorem]{Proposition}
\theoremstyle{definition}
\newtheorem{definition}[theorem]{Definition}

\newtheorem{remark}[theorem]{Remark}
\newtheorem{example}[theorem]{Example}
\theoremstyle{remark}

\newcommand{\res}
{\mathop{\hbox{\vrule height 7pt width .5pt depth 0pt \vrule
height .5pt width 6pt depth 0pt}}\nolimits}

\begin{document}

\title{John and uniform domains in generalized Siegel boundaries 
}

\author{Roberto Monti \thanks{Dipartimento di Matematica ``Tullio Levi-Civita'', Universit\`a degli Studi di Padova, (Italy)
 }  \and Daniele Morbidelli\thanks{Dipartimento di Matematica, Universit\`a di Bologna (Italy)}}

\date{}

 \maketitle


 \begin{abstract}
Given   the   pair of vector  fields    $X=\p_x+|z|^{2m}y\p_t$ and $ 
Y=\p_y-|z|^{2m}x \p_t,$
where $(x,y,t)= (z,t)\in\R^3=\C\times\R$, we give a condition on a bounded
domain $\Omega\subset\R^3$ which ensures that $\Omega$ is an $(\e
,\delta)$-domain  for the Carnot-Carath\'eodory metric.  We also
analyze the 
Ahlfors regularity of the   natural surface measure induced on  $
\partial \Omega$
by the vector fields. 
 \end{abstract}

%
 \section{Introduction}
 
In $\R^3=\C\times\R$ we consider  the vector fields
 \begin{equation}
 \label{eccoqui} 
  X=\p_x+|z|^{2m}y\p_t\qquad\text{and}\qquad Y=\p_y-|z|^{2m}x \p_t,
 \end{equation} 
where $(x,y,t)= (z,t)\in\R^3=\C\times\R$ and $m\in\left[1,+\infty\right[$ is a real
parameter. The  vector fields $X$ and $Y$ naturally arise as the real and
imaginary part of the  holomorphic  vector field tangent to the boundary of the
generalized Siegel domain  $\{(z_1,z_2)\in\C^2:\Im z_2>\frac{1}{2m+2}
|z_1|^{2m+2}\}$.

We study  the interaction of the
Carnot-Carath\'eodory (CC)  distance $d$ induced by $X$ and $Y$ with the
geometry
of a surface embedded in $\R^3$. Namely, we
give conditions on the boundary $\p\Omega$ such that an open set $\Omega
\subset\R^3$  is a John domain, a uniform domain and such that
  the natural surface measure induced on 
$\partial\Omega $ by 
$X$ and $Y$ is Ahlfors regular, see~Definition~\ref{MES}.

John domains are also known as domains with the \emph{twisted cone property},
see Definition~\ref{deffo}.
When the distance is induced by   H\"ormander vector fields in $\R^n$, 
several authors proved that a bounded John domain supports a 
global Sobolev-Poincar\'e inequality, see
\cite{Jerison,SaloffCoste92,FranchiLuWheeden96,GarofaloNhieu96} and
 the discussion for a general metric space in
\cite{HajlaszKoskela}. 
The exterior twisted cone property is also relevant in classical potential
theory
because it implies the subelliptic Wiener criterion (see
\cite{NegriniScornazzani}).

Uniform domains are also known as  $(\e,\delta)$-domains, see Definition
\ref{UNI}.
They form a subset of John domains. 
In the  global theory of Sobolev spaces for H\"ormander
 vector fields,  Garofalo and Nhieu
proved in \cite{GarofaloNhieu98} that subelliptic 
Sobolev functions in a uniform domain can be extended
to the whole space. In \cite{DanielliGarofaloNhieu06} 
it is also shown that the trace of a Sobolev function in a uniform domain with Ahlfors
regular boundary belongs to     a suitable Besov space of the boundary.
Also for this reason, we shall study the Ahlfors property very carefully.
The trace problem  was   analyzed in
\cite{MontiMorbidelli02} in the non-characteristic case and in a two-dimensional
model.
In  \cite{GerosaMontiMorbidelli}, we  
study by a direct approach the trace problem at the boundary of the
characteristic half plane $t>0$
for vector fields of  Martinet  type $X=\p_x$, $Y=\p_y+|x|^{\alpha}\p_t$ in
$\R^3$. 

In spite of the previous results, there are not  many examples
of  John and uniform domains in Carnot-Carath\'eodory spaces.
In fact, 
the subRiemannian case is  more delicate than the Euclidean one  because  of the
presence of \emph{characteristic points}, i.e.,~points  where   the
H\"ormander vector fields are all tangent to the boundary. Such points
make the construction of the inner cone more difficult. Sometimes the inner cone
does not exist at all, even for analytic boundaries, see 
e.g.~\cite[Theorem~1.2]{MontiMorbidelli05}.

  A well known general fact is that  small 
CC-balls are John domains.    We also mention some   contributions   by the russian school. See~\cite{VodopyanovGreshnov,Greshnov14,Greshnov18}, where   the authors study the uniformity of subRiemannian balls in Heisenberg groups.   See also \cite{Greshnov01}  and \cite{Romanovskii}, for further examples.  
In \cite{MontiMorbidelli05} it is proved that $C^2$ domains in Carnot groups
of step two are uniform. The case of cylindrically symmetric domains was
already considered in \cite{CapognaGarofalo98} in the Heisenberg group, that
is the model  \eqref{eccoqui} with $m=0$.
In \cite{MontiMorbidelli04} and
\cite{MontiMorbidelli05b}, the authors studied the case of diagonal vector
fields.

In this paper, we study uniform domains in $\R^3$ for the CC distance of the
vector fields~\eqref{eccoqui}. Our sufficient condition
for a domain to be uniform  requires the boundary  to be ``flat''
near characteristic points on the $t$-axis.

Let $\Omega\subset\R^3$ be an open set with $C^\infty$ boundary.
If both $X$ and $Y$ are tangent to $\partial \Omega$ at the point $p
\in\partial \Omega$, then there is a neighborhood $U_p$ of
$p$ such that $U_p\cap\partial\Omega$ is a graph of the form $t = \phi(z)$.
So we start from the following definition.

Let $A\subset \R^2$ be an open set and $\phi \in C^\infty(A)$. 
We say that
$\Sigma =\mathrm{gr}(\phi) =\{(z,\phi(z)) \in\R^3:z\in A\}$ is an
\emph{$m$-admissible
graph} if there exists a constant $C>0$ such that for all $z\in A$
\begin{equation}\label{terzina} 
  \abs{D^3\phi(z)}
 \leq C|z|^{2m-1},\quad
 \abs{D^2\phi(z)}\leq C|z|^{2m}\quad\text{ and }\quad \abs{D\phi(z)}\leq
C|z|^{2m+1}.
\end{equation} 
When $0 \notin A$, the three conditions \eqref{terzina} are trivially
satisfied in a compact subset of $A$. The conditions are instead   
restrictive when $0\in A$.   In \eqref{terzina} we adopted the notation $|D^k\phi(z)|:=\max_{j_1, \dots, j_k}|\frac{\p^k\phi }{\p x_{j_1} \cdots \p x_{j_k}}(z) |$ to denote the largest $k$-th order derivative of~$\phi$.

 \begin{definition}  Let  $m\in [1,+\infty[$. We say that a bounded domain
$\Omega\subset \R^3$
with smooth boundary is $m$-admissible 
if for any characteristic point $p\in\partial \Omega$ there exists a
neighborhood   $U_p$ of $p$ in $\R^3$ such that $\partial \Omega\cap
U_p$ is an
$m$-admissible graph.
\end{definition}

 It is easy to construct simple examples of admissible sets.  Consider for instance the bounded domain  \color{black} 
\begin{equation}\label{ammm} 
 \Omega=\{(z,t)\in\R^3: |z|^{2(m+1)}+t^2=1\}.
\end{equation} The boundary  $\p\Omega$ has two characteristic points, namely $
(0,0,-1)$ and $(0,0, 1)$, and the functions $\phi(z)=\pm \sqrt{1-|z|^{2(m+1)}\;}$ satisfy condition~\eqref{terzina}.  Small perturbations of the boundary, compactly supported outside the characteristic set, give  nonradial examples.

Our main result is the following:

\begin{theorem}\label{principale} 
Let $m\in\N$.    Any
$m$-admissible domain $\Omega\subset\R^3$
is   uniform  and, in particular,   is a John domain in $(\R^3,d)$.
\end{theorem}

In fact, our proof shows that admissible domains are also non-tangentially
accessible (NTA).  Concerning our requirements on the rate of 
growth   \eqref{terzina} for the function $\phi$, it is 
easy to check that any open set which agrees in 
a neighborhood of the origin with  the epigraph    $\{ t>|z|^{\a}\}$
with $\a< 2m+2$
 is not a John   domain.

On the other hand, let us  consider the epigraph $\{ t>-x^{2m+1}y\}$ 
of Example~\ref{esempietto}. All the points  $(x,0,0)$  of the $x$-axis are
characteristic points of the boundary. 
However, the ``order of degeneration'' of such points is  2 when $x\neq 0$,
while it is $2m+2$ when $x=0$. The difficulty of our work in Section \ref{CONO}
is due to the fact that we need to construct   a family of inner cones  of
\emph{constant aperture} contained in  $\{ t>-x^{2m+1}y\}$ and with
vertex at points arbitrarily close to the characteristic set.
Furthermore, in order to prove the $(\e,\d)$-property, in
Section~\ref{disconosco} we  also need to show that cones with close vertices
  have quantitatively close axes.

Theorem 
\ref{principale} is proved in Sections~\ref{CONO},~\ref{disconosco} and~\ref{ultimissima}.
We first show that  global (i.e., with $A=\R^2$)
admissible  graphs have the global cone property  and  then that they satisfy
the
$(\e,\delta)$-condition. Finally, we deal with the case of bounded
domains.  
The proofs rely on a precise description of the
distance $d$, which will be discussed in Section~\ref{giovanni}, and on some
preliminary results proved in Section~\ref{prel}.

The natural surface area on $\partial \Omega$ is the perimeter measure of
$\Omega$ induced by the vector fields \eqref{eccoqui}. This is the measure
\begin{equation} \label{MES}
   \mu = \sqrt{\langle N,X\rangle ^2 + \langle N,Y\rangle^2\; }\, \mathcal H^2 \res
\partial \Omega,
\end{equation}
where $N$ is the unit Euclidean normal to $\partial \Omega$,
$\langle\cdot,\cdot\rangle$ is the standard scalar product of $\R^3$, and 
$\mathcal H^2 \res
\partial \Omega$ is the standard surface measure, i.e., the restriction of the
$2$-dimensional Hausdorff measure to $\partial\Omega$. This is a special case of
the
   variational definition of perimeter measure
in CC-spaces, see \cite{GarofaloNhieu98} and \cite{MontiSerraCassano}.
For admissible domains, the measure $\mu$ is codimension $1$ Ahlfors regular in
the following sense.

\begin{theorem} \label{Ahl} 
Let $m\in\N$  and denote by $B(p,r)$  
the CC-balls. 
For any  $m$-admissible domain $\Omega\subset\R^3$ there exist  constants
$C>0$ and $r_0>0$ such that for all $p\in\p\Omega$ and $0<r\leq r_0$
 \begin{equation}
  \label{Renzi}
  C^{-1}\frac{|B(p,r)|}{r}\leq \mu(B(p,r))\leq C\frac{|B(p,r)|}{r}.
 \end{equation} 
\end{theorem}

Above,   $|\cdot|$ denotes the Lebesgue measure in $\R^3$.
This theorem is proved in Section~\ref{ultimissima} and relies 
 on the delicate analysis  of  
 global admissible graphs tackled in Section~\ref{DiMaio}. Our analysis  will
require the   study of several
situations,  depending on 
how CC-balls  intersect the graph near the  characteristic set.

The  ball-box theorem for the distance $d$ is proved in  the first part of the paper.
For any $(z,t),(\z,\t)\in\R^3$, we define
the function
 \begin{equation} 
  \label{distanza} 
  \d((z,t),(\z,\t)) = \abs{z-\z}
  +
  \min
  \bigg\{ 
 |v|^{\frac{1}{2m+2}},
  \frac{     | v|^{1/2}}{ \abs{z}^m } 
\bigg\},
 \end{equation} 
where $v= \t-t+ \abs{z}^{2m}\omega(z,\z) $, and $\omega(z,\z) =    x\eta-y\xi
$ with $z=(x,y)$ and $\z=
(\xi,\eta)$.

 \begin{theorem}\label{eqq}  Let $m\in[1,+\infty[$.
There is a constant $C>1$ such that for all $p=(z,t),q=(\z,\t)\in\R^3$
\begin{equation}
\label{cinquetre} 
C ^{-1}\d(p,q)\leq d(p,q) \leq C  \delta
(p,q).                        
\end{equation}  
 \end{theorem}

This theorem is proved in Section \ref{giovanni}.
    Our proof is completely self-contained and works for
any $m\geq 1$, also noninteger. Note that when $m\in
[1,+\infty[\setminus\N$, 
the well known ball-box theorem in \cite{NagelSteinWainger} cannot be applied, because
the vector fields  \eqref{eccoqui} are not smooth at $z=0$. 
 In the case $m\in\N$,  a local version of Theorem~\ref{eqq} can be obtained from the classical results in~\cite{NagelSteinWainger}. The statement can in principle  be globalized by some dilation argument, but this requires some care. Here, we give an independent self-contained proof of Theorem~\ref{eqq}. In particular,  Step~2 and Step~3 of the proof of this theorem give a constructive  and quantitative explanation of the fact that any pair of points can be connected with a  horizontal path. 
\color{black}

\section{Ball-box estimate}

\label{giovanni} 

In this section, we prove Theorem \ref{eqq}  and, in Corollary \ref{scatola}
below, we rephrase it as a ball-box estimate. 

An absolutely continuous curve
$\gamma:[0,1]\to\R^3$ is \emph{horizontal} for the vector fields
\eqref{eccoqui},
if it satisfies $\dot\gamma=\alpha(s) X(\gamma) +\beta(s) Y(\gamma)$ for
a.e.~$s\in[0,1]$. 
The length of  $\gamma$ is defined as 
 \begin{equation*}
  \operatorname{length}(\gamma) =\int_0^1\abs{ (\alpha (s), \beta (s)) } ds.
 \end{equation*}
Given  points $(z,t),(\z,\t)\in\R^3$,  the CC distance $ d((z,t),(\z,\t))$ is
defined  as the infimum (the minimum, in fact) of the length  of all absolutely
continuous curves $\gamma:[0,1]\to\R^3$ connecting them.

We will   use  the following invariance properties of~$d$.  
For all  $(z,t),(\z,\t)\in\C\times\R$, $s,\theta\in\R$, and  $r>0$ we have:
\begin{align}
\label{rotoliamo}
&d((z,t),(\zeta,\t) )  =d (e^{i\theta}z,t), (e^{i\theta}\z,\t)) ; 
\\ 
\label{rotoliamo2}
& d((z,t), (\zeta,\t))=d((z,t+s),(\z,\t+s)); 
 \\
& d\big((r z,r^{2m+2}t),(r\z,r^{2m+2}\t))=r d((z,t),(\zeta,\t) ) .
\end{align}
We will also use the following elementary estimate,   holding for any $x,y\in\R$
and $m\geq 1$: 
 \begin{equation}
  \label{pillo}
C_{m}^{-1}(\abs{x}^{m-1}+\abs{y}^{m-1})\;\big||x|-|y|\big| \leq \big|\abs{x}^m -\abs{y}^m \big| \leq
C_{m}(\abs{x}^{m-1}+\abs{y}^{m-1})\abs{y-x} .
 \end{equation}

\begin{proof}[Proof of Theorem \ref{eqq} ]
\step{Step 1.} We claim that there exists a constant $C>0$, depending on
$m$, such that    $\delta((z_0,t_0),(\z,\tau) ) \leq C d((z_0,t_0),(\z,\tau))
$ for all points $(z_0,t_0),(\z,\tau)\in\R^3$.

By~\eqref{rotoliamo}-\eqref{rotoliamo2}, we can assume that $z_0= (x_0,0)$ with
$x_0\geq 0$ and $t_0=0$. In this case, we have $\omega(z_0,\z) = x_0 
\eta$,
with $\z=(\xi,\eta)$, and the definition in \eqref{distanza} for $\delta$ reads,
with $v = \t+x_0^{2m+1}\eta$, 
\begin{equation*}
 \delta((z_0,0), (\zeta,\t)) =  \abs{z_0- \z}
 +\min\Big\{ \frac{\abs{v}^{1/2}}{x_0^m},
 \abs{v}^{\frac{1}{2m+2}}\Big\}.
\end{equation*}

Let $\gamma=(z,t) :[0,T]\to\R^3$, $T>0$, be a unit-speed horizontal curve
connecting $(z_0,0)$ and $(\zeta,\tau)$. We let $z= z(s) = (x(s),y(s)) = (x,y)$.
From the unit-speed condition  $|\dot z|\leq 1$, we deduce that 
\begin{equation}
 \label{unox}
|z_0-\z| =\Big| \int_0^T \dot z \, ds\Big| \leq T.
\end{equation}
We estimate the quantity
\[
 v = \t+x_0^{2m+1} \eta = \int_0^T \Big\{ |z|^{2m} y\dot x +(x_0^{2m+1}
-|z|^{2m}
x) \dot y\Big\} ds.
\]
We claim that there exists a constant $C>0$ such that for all $s\in [0,T]$ we
have
\begin{equation} 
 \label{pax}
      |z|^{2m} |y|  +|x_0^{2m+1} -|z|^{2m}x| \leq C (x_0^{2m} s + s^{2m+1}).
\end{equation}
The left-hand side is evaluated at $s\in[0,T]$.
From $ |z| \leq x_0+ s$ and   $|y|\leq s$ we deduce that 
$|z|^{2m} |y| \leq C (x_0^{2m} s + s^{2m+1})$.
By the triangle inequality
and \eqref{pillo}, we have
\[
\begin{split}
  \big|x_0^{2m+1} -|z|^{2m}x\big| & 
  \leq \big| |z|^{2m} -x_0^{2m} \big | |x| + x_0^{2m} |x-x_0|
\\
&
\leq C_m  ( |z|^{2m-1} +x_0^{2m-1}) \big| |z|-x_0 \big| |x| + x_0^{2m}
|x-x_0|.
\end{split}
\]
Using $|x|\leq |z| \leq x_0+ s$, $\big||z|-x_0\big|\leq s$ and $|x-x_0|\leq s$ 
we   obtain $ \big|x_0^{2m+1} -|z|^{2m}x\big| \leq C (x_0^{2m} s + s^{2m+1})$.
This finishes the proof of \eqref{pax}.
 
Now, \eqref{pax} implies that $\abs{\t+x_0^{2m+1}\eta}\leq  C (x_0^{2m} T^2 +
T^{2m+2})$, 
which is equivalent  to  
 \begin{equation} \label{mia}
 T\geq C^{-1} \min \Bigl\{\frac{\abs{\t+x_0^{2m+1}\eta}^{1/2}}{x_0^m}, 
 \abs{\t+x_0^{2m+1}\eta}^{\frac{1}{2m+2 }}  \Bigr\} .
 \end{equation}
The inequalities \eqref{mia} and \eqref{unox} 
imply $ \delta((z_0,0), (\zeta,\t))  \leq C T$ and minimizing on $T$ we get
the claim made in the Step 1.

\medskip

\step{Step 2.} We claim that there exists a constant $C>0$ such that 
$d((z,t),(z,\t))\leq C \d((z,t),(z,\t))$ for all $z\in\C$ and $t,\t\in\R$, i.e.,
  \begin{equation}
\label{lemmao} 
   d((z,t),(z,\t))\leq C  \min\Big\{\abs{\t-t}^{\frac{1}{2m+2}},
   \frac{\abs{\t-t}^{1/2}}{\abs{z}^m}\Big\}.
  \end{equation}
By \eqref{rotoliamo}--\eqref{rotoliamo2}, we can without loss of generality
assume that $z= (x,0)$ with $x\geq 0$, $t=0$ and  $\t\geq 0$.

For each $u\geq 0$ consider
 the unit-speed  path   $[0, 4u]\ni s\mapsto \z(s)\in \R^2$ that
linearly connects the
points in the plane  $(x,0)$, $(x,u)$, $(x+u, u)$, $(x+u, 0)$ and $(x,0)$.
Let $R_u$ be the square enclosed by $\z$. The path $\z$ has length $4u$ and its
unique absolutely continuous  horizontal  lift  $s\mapsto
\gamma(s) = (\z(s),\t(s))$
satisfying $\tau(0)=0$
has final point
\begin{equation}
 \label{stella}
 \t(4u)=\int_\z \abs{\z}^{2m}(\eta d\xi-\xi d\eta) 
 =2(m+1)\int_{R_u}\abs{\z}^{2m}d\xi d\eta
 \geq C_0^{-1} u^2(x^{2m}+u^{2m}  ).
\end{equation}
We used  Stokes' 
theorem with the
counterclockwise orientation of $\z$. 
The function $u\mapsto \int_{R_u}|\z|^{2m}d\xi d\eta$ is a strictly increasing
bijection of $\left[0, +\infty\right[$ onto itself.

Let $\ol u$  be the unique number such that $\tau( 4\ol u )=\tau$. By the definition of the distance $d$ and by \eqref{stella}, we have  
\begin{equation*}
\begin{aligned}
 d((x,0,0), (x,0,\t)) & \leq  4\ol u =
 \min\big\{4u>0:  \tau(4 u) 
\geq 
\tau  \big\}
\\
&
 \leq   \min \big\{4u >0:\t\leq  C_0^{-1} u^2(x^{2m}+u^{2m} )  \big\}
\\
&
\leq C  \min\Big\{ \frac{\t^{1/2}}{x^m}, \t^{\frac{1}{2m+2}}\Big\}.
\end{aligned}
\end{equation*}
This concludes the proof of the  Step 2.

\medskip

\step{Step 3.}   We claim that the inequality  $d((z,t),(\z,\t))\leq C
\delta((z,t),(\z,\t))$ holds for all points $(z,t),(\z,\t)\in\C\times\R$.

We  preliminarily observe that, given $(u,v )=w\in\C$, for any point $(z,t) =
(x,y,t)$ we have 
\begin{equation*}
e^{uX+vY} (z,t)   =\Bigl(z+w, t+\omega(w,z)  \int_0^1\abs{z+s w}^{2m}ds 
 \Bigr),
\end{equation*}
where $ \omega(w,z) =  uy-vx $ and $e^{uX+vY} (z,t)$ denotes the value at time
$1$ of the integral curve 
 of $uX+vY$ starting from $(z,t)$ at time $0$.

By the triangle inequality, it follows that  
\begin{equation*}
\begin{aligned}
   d((z,t),  (\z,\t)) 
\leq d\big((z,t ), e^{(\xi-x)X+(\eta-y) Y}(z,t)\big) +
 d\big( e^{(\xi-x)X+(\eta-y) Y}(z,t), (\z,\t) \big).
\end{aligned}
\end{equation*}
In the last distance, the points are one above each other and so, by
\eqref{lemmao}, we get
\begin{equation}
 \label{NUM} 
\begin{aligned}
   d((z,t),  (\z,\t)) 
\leq   C\Big( \abs{z-\z}+
\min\Big\{
 |t-\tau +\lambda  |^{\frac{1}{2m+2}}
, 
 { |t-\t+\lambda |^{\frac 12}}/
{\abs{ \z}^m}
\Big \} 
 \Big),
\end{aligned}
\end{equation}
where
\[
 \lambda = 
\omega
(\zeta,z)  \int_0^1  \abs{z+s(\z-z )}^{2m}ds.
\]
We used $\omega(\z-z,z) = \omega(\z,z)$.

In order to prove the claim in the Step 3, we have to show that the right-hand
side in \eqref{NUM} is less than $C \delta((z,t),(\z,\t))$.
By \eqref{rotoliamo}--\eqref{rotoliamo2}, it is enough to prove this estimate
in the case $z=(x,0)$ with $x\geq 0$ and $t=0$. In this case, the distance
$\delta$ is
\[
\delta((z,0),(\z,\t)) = 
   |z-\z| +
  \min
  \bigg\{
 | v |^{\frac{1}{2m+2}},
  \frac{     |   v |^{1/2}}{x^m } 
\bigg\}  ,
\qquad 
 v = \t+ x^{2m+1}\eta.
\] 
We distinguish two cases:

\step{Case G1:} $\abs{v}^{\frac{1}{2m+2}}\leq \abs{v}^{\frac 1 2}/ {x^m} $,
i.e.,
$x\leq \abs{v}^{\frac{1}{2m+2}}$;
 
 \step{Case G2:} $\abs{v}^{\frac{1}{2m+2}}\geq  \abs{v}^{\frac 1 2}/{x^m}$,
i.e.,
$x\geq  \abs{v}^{\frac{1}{2m+2}}$.

\medskip

 \step{Case G1.} When $z=(x,0)$ and $t=0$ we have $\omega(\z,z) = - x\eta$ and,
see \eqref{NUM},  
\[
 |t-\t+\lambda| = \Big|\t+  
x\eta\int_0^1 \abs{z+s(\z-z)}^{2m}ds
\Big|.
\]
We claim that in the Case G1 we have 
\begin{equation}
\begin{aligned}\label{novecinquedue} 
\Big|\t+  
x\eta\int_0^1 \abs{z+s(\z-z)}^{2m}ds
\Big|^{\frac{1}{2m+2}}
&
 \leq C\Big(    \abs{\z-z}+  \abs{v}^{\frac{1}{2m+2}}\Big).
\end{aligned}
\end{equation}
This and \eqref{NUM} finish the proof of the the Step 3 in the Case G1, because
the right-hand
side in \eqref{novecinquedue} is  
$C \delta((z,0),(\z,\t)) $.

We prove \eqref{novecinquedue}.
By the triangle inequality, we have
\begin{equation}
 \label{spiezzo} 
\begin{split}
 \Big|\t+  
& x\eta\int_0^1 \abs{z+s(\z-z)}^{2m}ds \Big|
\leq \abs{v}+ x \abs{\eta}  
\Big| -x^{2m}+
\int_0^1 \abs{z+s(\z-z)}^{2m}ds \Big|
\\&=
\abs{v}+
x\, \abs{\eta} 
\Big| 
\int_0^1  \int_0^s  \frac{d}{d\rho}\abs{z+\rho(\z-z)}^{2m} d\rho \, ds
\Big|
\\&
=
\abs{v}+
2mx\, \abs{\eta} 
\Big| 
\int_0^1 \int_0^s\abs{z+\rho(\z-z) }^{2m-2}
\big\{ \langle z,\z-z\rangle +\rho|\z-z|^2    \big\}  d\rho  \, ds
\Big|
\\&
\leq  C( \abs{v}+\Theta ),
\end{split}
\end{equation} 
where we let
\begin{equation} \label{pacco}
 \Theta =x  \abs{\eta} 
  (x+\abs{\z-z})^{2m-2}
(x\abs{\xi-x} +|\z-z|^2    ).
\end{equation}
By the H\"older inequality and by the Case G1 we have
\[
\begin{split} 
 \Theta & \leq C x  \abs{\z-z} 
  (x+\abs{\z-z})^{2m}
    \leq C  (x+\abs{\z-z})^{2m+2}  
\leq C(|v|   +\abs{\z-z} ^{2m+2}).
\end{split}
\]
This and \eqref{spiezzo} finish the proof of \eqref{novecinquedue}.

\medskip

\step{Case G2.}  In this case we have  $x\geq \abs{v}^{\frac{1}{2m+2}}$, and
thus
\[
  \delta((z,0),(\z,\t)) =|z-\z|+\frac{\abs{v}^{1/2}}{x^m},
\qquad 
 v = \t+ x^{2m+1}\eta.
\]
We distinguish the following
three subcases:
\begin{align*}
&  \abs{\xi-x}\leq \frac 12 x;
 \tag{G2a}
\\ 
&   
\max\{\abs{\xi },\abs{\eta}\}\leq \frac 12 x;  \tag{G2b}
\\
 &   \max\{\abs{\xi },\abs{\eta}\}\geq  \frac 12 x
        \quad\text{and}\quad \abs{\xi-x}\geq \frac 12 x.
   \tag{G2c}
\end{align*}

In the Case G2a, 
the quantity $\Theta$ in \eqref{pacco} can be estimated as follows
\[
  \Theta  \leq C x  \abs{\eta} 
  (x+\abs{\eta })^{2m-2}
(x |x-\xi| +\eta^2 )  
\]
and from $ x\leq  x+\abs{\eta}  \leq  C|\z|$ we deduce that 
\[ 
 \frac{1}{|\z| ^{2m} } (\abs{v} +\Theta)
 \leq C \Big( \frac{|v|}{x^{2m}} +  \frac{ x 
\abs{\eta}}{ 
  (x+\abs{\eta })^{ 2}}
(x |x-\xi| +\eta^2 ) \Big )\leq C\Big( \frac{|v|}{x^{2m}} +  |\zeta-z|^2
\Big).
\]
This along with \eqref{spiezzo} and \eqref{NUM}
finish the proof of the Step 3 in the Case G2a.

 In the Case G2b, the quantities 
$x$, $x+\abs{\xi}$, $ x+\abs{\eta}$, $\abs{\xi-x} $ are mutually  
comparable with absolute constants and therefore, also using the H\"older
inequality, the quantity $\Theta$ in
\eqref{pacco} can be estimated as follows
\[
\Theta \leq C  x^{2m+1}\abs{\eta} \leq C |\z-z|^{2m+2} .
 \]
On the other hand, we have
\[
\abs{v}^{\frac{1}{2m+2}} \leq C
\Big (x+\frac{\abs{v}^{1/2}} {x^m}\Big) \leq C \Big(|\z-z|
+\frac{\abs{v}^{1/2}} {x^m}\Big ).
\]
These two inequalities imply, via \eqref{spiezzo}, that the claim
\eqref{novecinquedue} holds also in the Case~G2b.

 In the Case  G2c,    we have $x\leq C|\z| $ and from \eqref{pacco} we
estimate
\[
 \Theta \leq C  \abs{\z}   \abs{\eta} 
  (\abs{\z} +\abs{\z-z})^{2m-2}
(\abs{\z} \abs{\z-z} +|\z-z|^2    ) \leq C \abs{\z} ^{2m} \abs{\z-z}^2,
\]
and we conclude that
\[
 \frac{1}{\abs{\z}^{m}} (\Theta+|v|)^{1/2}  \leq C \Big( \abs{\z-z} +
\frac{|v|^{1/2}}{x^{m}}\Big).
\]
The proof is concluded also in this case.
\end{proof}

\begin{remark}\label{diametrale} 

The   argument of the proof of  Theorem
~\ref{eqq} shows  in fact the  
global
equivalence
\begin{equation}
\label{diamett} 
|(u,v)|\leq d ((z,t),  e^{uX+vY}(z,t))\leq  C_0|(u,v)|,\quad  \quad  
u,v,t\in\R, \quad z\in\R^2.
\end{equation} 
\end{remark}

Next we describe the   $d$-balls as suitable boxes.
For any  $\beta>0$ we define the weighted norm of $u=(u_1, u_2, u_3)\in\R^3$  
\[
 \norm{u}_{1,1,\beta}=\max\{|u_1|,\abs{u_2},\abs{u_3}^{1/\beta}\},
\]
and for any $p=(z,t) =(x,y,t) \in\R^3$ and $r>0$ we define the boxes
\begin{equation*}
\begin{aligned}
&  \Box_I (p, r) = \Big\{ \big( x+u_1, y+ u_2, t+ \abs{z}^{2m}(u_3+yu_1-
xu_2)\big):  \norm{u}
  _{1,1,2}< r  \Big\},
\\&
  \Box_J (p, r) = \Big\{\big( x+u_1, y+ u_2 ,t+ u_3+ \abs{z}^{2m}( yu_1-
xu_2)\big) 
  :  \norm{u}_{1,1,2m+2}< r \Big\}.
\end{aligned}
\end{equation*}

\begin{corollary}\label{scatola} Let $m\in [1,+\infty[$.
For any  $\a>0$ there exist constants  $b_1 , b_2 ,
 \d_0 >0$ such that for all $p=(z,t) \in\R^3$ and $r>0$ we have:
 
 (i) if $|z|\geq \a r$, then 
 \begin{equation}\label{giglio} 
  \Box_I(p, \d_0   r)\subset B(p, r)\subset \Box_I(p, b_1 r);
 \end{equation} 
 
(ii)  if  $r\geq \a|z|$, then 
 \begin{equation}\label{collo} 
  \Box_J(p,  \d_0  r)\subset B(p, r)\subset \Box_J(p, b_2  r).
 \end{equation} 
\end{corollary}

 \begin{proof} 
  \step{Step 1.} We claim that for a suitable $\d_0>0$ we have
  \begin{equation*}
   \Box_I(p,  \d_0 r)\cup \Box_J(p, \d_0 r)
   \subset B(p , r)
   \subset \Box_I(p, r/ \d_0 )\cup \Box_J(p,  r/\d_0 ).
  \end{equation*}
To prove these inclusions, we observe that, letting $v = \t - t +|z|^{2m}(x\eta-
y\xi)$,
\begin{align}
 \label{giglio1}
 (\z,\t)\in \Box_I (p, r) & \quad\Leftrightarrow  \quad
 \max\bigg \{|\xi-x|,|\eta-y|,\frac{\abs{v}^{1/ 2  }
}{|z|^{m}}\bigg\}<r ,
\\
 \label{collo1}
 (\z,\t)\in \Box_J (p, r)
 & \quad\Leftrightarrow  \quad
 \max\bigg \{|\xi-x|,|\eta-y|, \abs{v}^{\frac{1}{2m+2 }} \bigg \}<r .
\end{align}
Thus the point $(\z,\t)$ belongs to the union of the boxes if and only if 
$|\xi-x|<r$, $|\eta-y|<r$ and 
\begin{equation*} \min\bigg\{ 
 \abs{v}^{\frac{1}{2m+2}} , \frac{\abs{v}^{\frac 12 } }{|z|^m}
 \bigg\}   <r.
\end{equation*}
Now the claim follows from Theorem  \ref{eqq}.  We also 
proved both the  inclusions in the left-hand side of \eqref{giglio} and
\eqref{collo}.

 \step{Step 2.} We  prove  the  inclusion in the right-hand side
of \eqref{giglio}. Let $|z|\geq \a r$ and let $(\z,\t)\in B(p, r)$. By the Step
1 we
know that $(\z,\t)\in (B_I\cup B_J)(p, r/ \d_0)$. Then, we are left
to show that  
 \begin{equation*}
  |z|\geq \a r\quad\text{and}\quad \min\Big\{|v|^{\frac{1}{2m+2}},
\frac{|v|^{1/2}}{|z|^m}
  \Big\}<  r/\d_0 \quad  \Rightarrow \quad  
\frac{\abs{v}^{1/2}}{|z|^m}< b_1  r. 
 \end{equation*}
If the minimum is $ {|v|^{1/2}}/ {|z|^m}$, there is nothing to prove.
Otherwise we have $|v|^{1/(2m+2)}< r/ \d_0 $, i.e., $|v|^{1/2}<
(r/ \d_0)^{m+1}$. This and $|z|\geq \a  r$ imply
 \begin{equation*}
  \frac{|v|^{1/2}}{|z|^m}\leq \frac{|v|^{1/2}}{(\a  r)^m}\leq \frac{(
r/ \d_0 )^{m+1}}{(\a r)^m}
  = \frac{1}{\a^m \d_0^{m+1}}r.
 \end{equation*}

 \step{Step 3.} We prove the inclusion \eqref{collo}. Arguing as in the Step 2,
it
suffices to prove that
 \begin{equation*}
  r\geq\a |z|\quad \text{and}\quad \min\Big\{|v|^{\frac{1}{2m+2}},
\frac{|v|^{1/2}}{|z|^m}\Big
  \}< \frac{r}{ \d_0} 
   \quad
  \Rightarrow 
\quad  
|v| ^{\frac{1}{2m+2}}\leq b_2  r. 
 \end{equation*}
If the minimum is $|v|^{\frac{1}{2m+2}}$, there is nothing to prove.
Otherwise we have 
\[ \frac{r}{ \d_0} \geq \frac{|v|^{1/2}}{|z|^m}\geq
|v|^{1/2}\Big(\frac{\a}{r}\Big)^m,
\]
that is  equivalent to $|v|^{1/2}\leq r^{m+1} /\d_0 \a^m$. This is the
claim.
  \end{proof}


\section{Geometry of admissible graphs}\label{prel} 

Let $\phi \in C^\infty(\R^2)$ be a smooth
function satisfying the
flatness conditions \eqref{terzina} at any point $z\in \R^2$.
A defining function for the graph of $\phi$ is  the function $F\in
C^\infty(\R^3)$ given by  $F(z,t)= \phi(z)-t$.
The derivatives  
\begin{equation*}
\begin{split}
  XF(z,t)& =XF(z)=\phi_x(z)-\abs{z}^{2m}y ,
\\ 
YF(z,t)& = YF(z)= \phi_y (z)+\abs{z}^{2m}x
\end{split}
 \end{equation*}
do not depend on $t$ and we let 
$ZF(z)=(XF(z),YF(z))$.
A point $(z,\phi(z)) \in \Sigma = \mathrm{gr}(\phi)$ is characteristic if and
only if
$ZF (z)=0$.
By \eqref{terzina}, the function $ZF$ satisfies  
\begin{equation}\label{wow} 
  |ZF(z)| \leq C \abs{z}^{2m+1},  \quad\  z\in\R^2.  
\end{equation}

\begin{example}\label{esempietto}   Let $m\in \N$. 
The graph of the function $\phi(z) = - x^{2m+1} y$   is
$m$-admissible
and  each point of the $x$-axis is characteristic.
\end{example}

The next proposition describes the restriction of the distance $d$ to an
admissible graph.

 \begin{lemma}
 \label{lascrivo}  Let $\phi \in C^\infty(\R^2)$ satisfy the conditions
\eqref{terzina}. Then there exist a constant $C_0>0$ such that  for all
$p=(z,\phi(z)), q = (\z,\phi(\z)) \in\Sigma $
\begin{gather}
 \label{gallo}
C_0^{-1} d(p,q)\leq \abs{\z- z }+
  \Big| 
  \frac{  \phi(\z )-\phi(z )}{   \mass ^{2m}}  + \omega(z,\z)  \Big|^{\frac
1 2} 
  \leq C_0 d(p,q),
 \end{gather}
 where $\mass = \max\{|z|,|\z| \} $ and $\omega(z,\z) = x\eta-y\xi$.
\end{lemma}

\begin{proof} Without loss of generality, we prove the lemma in the case
$\abs{z}\geq\abs{\z}$.  We claim that, letting $v=   \phi(\z)-\phi(z)+
\abs{z}^{2m}
\omega(z,\z)$, we have 
\begin{equation} \label{cal}
 \frac{|v|^{1/2}}{\abs{z}^m} \leq C (|\z-z|+ | v |^{\frac{1}{m+2}})
\end{equation} 
Taking \eqref{cal} for granted  and starting  from Theorem \ref{eqq}, by
\eqref{cal} it follows that 
 \begin{equation*} 
 d(p, q)\simeq \abs{\z -z}
 + \min\Bigl\{  
   | v |^{\frac{1}{m+2}}
+
  \frac{|v|^{1/2}}{\abs{z}^m}
   \Bigr\} \simeq |\z-z|+  \frac{|v|^{1/2}}{\abs{z}^m} ,
 \end{equation*}
and this is \eqref{gallo}.

We prove \eqref{cal}. From the elementary inequality 
$a^{1/2}b^{1/2}\leq C 
 (a^{ \frac{2m+1}{2m+2}}b^{\frac{1}{2m+2}} +b)$ for   $a,b\geq 0$,
we obtain 
\begin{equation*}
\begin{split}
 \frac{|v|^{1/2}}{ \abs{z}^m}
& =\abs{z}^{1/2} 
\Big|
\frac{\phi(\z)-\phi(z)} { \abs{z}^{2m+1}} + \frac{\omega(z,\z) }{\abs{z}}
\Big|^{1/2}
\\
&
\leq C \Big\{ 
\abs{z}^{\frac{2m+1}{2m+2}} 
\Big|
\frac{\phi(\z)-\phi(z)} {\abs{z}^{2m+1}} + \frac{\omega(z,\z)   }{\abs{z}}
\Big|^{\frac{1}{2m+2}} 
+
\Big|\frac{\phi(\z)-\phi(z)} { \abs{z}^{2m+1}}  + \frac{\omega(z,\z)  
}{\abs{z}}
\Big| \Big\} 
\\
&
 \leq C \Big\{ 
 |v|^{\frac{1}{2m+2}} +
 \frac{|\phi(\z) -\phi(z)|} {\abs{z}^{2m+1}} 
+ \frac{|\omega(z,\z)  |}{\abs{z}}
  \Big\} 
\\ 
&
 \leq C \Big\{ 
 |v|^{\frac{1}{2m+2}} + |z-\z| 
  \Big\} ,
\end{split}
\end{equation*}
because by \eqref{terzina} and $|\z|\leq |z|$ we have 
\begin{equation*} 
 \abs{\phi(\z) -\phi(z)}     \leq  C
(\abs{z}^{2m+1}+\abs{\z}^{2m+1})\abs{\z-z}\leq 
 C  \abs{z}^{2m+1} \abs{\z-z}, 
\end{equation*}
and, moreover, $|\omega(z,\z)| = |\omega(z,\z-z)|\leq |z| |\z-z|$.
\end{proof}

In the next propositions, we  discuss other consequences of the conditions~\eqref{terzina}.

\begin{proposition}
Let $\phi\in C^\infty(\R^2)$ satisfy \eqref{terzina}.
For all $z,\z\in\R^2$ we have  
\begin{equation}\label{macchia2} 
\begin{aligned}
 \phi(\z)-\phi(z) +\mu_{\z,z}^{2m}\omega(z,\z) &
=\scalar{ZF(z)}{\z -z} 
 +\mass^{2m}O(\abs{\z-z}^2),\end{aligned}
\end{equation}  
where   $\mass = \max\{|z|,|\z| \}$ 
and the remainder satisfies 
$|O(\abs{\z-z}^2)| \leq C |z-\z|^2$ for a constant
$C>0$.
\end{proposition}

\begin{proof}
Expanding  $\phi$  at the second order at a point $z\in\R^2$, we obtain for any
$\z\in\R^2$ 
\begin{equation}\label{marrone} 
 \begin{aligned}
\phi(\z) & - \phi(z)   +\abs{z}^{2m} \omega(z,\z)  
= \langle ZF(z),\z-z\rangle  +\mass^{2m}
O(\abs{\z-z}^2).
\end{aligned}
\end{equation} 
By \eqref{terzina}, the remainder satisfies the uniform
estimate $\abs{O(\abs{\z-z}^2)}\leq C\abs{\z-z}^2$ for all $z,\z\in\C$. 
If $\mass = \abs{z} $, this is our claim \eqref{macchia2}.

If
$\abs{z}\leq \abs{\z}$,  starting from 
\eqref{marrone} it suffices to use the estimates $\abs{\omega(z,\z) }\leq 
  \abs{\z}\abs{\z-z}$ and 
\begin{equation*}
\begin{aligned}
 |\abs{\z}^{2m}-\abs{z}^{2m}|&\leq C \mu_{\z,z}^{2m-1}\abs{\z-z}\leq C 
\abs{\z}^{2m-1}
 \abs{\z-z} .
\end{aligned}
\end{equation*} 
and the proof is concluded. 
\end{proof}

Next we get a Taylor expansion  of $ ZF(\z)  $ with a remainder
$O(\abs{\z-z}^2)$. This is the only point where we use the assumption
$\abs{D^3\phi(z)}\leq C \abs{z}^{2m-1}$.

Let  $z,\z\in\R^2$ be points with $  |\z|\leq 2|z|$.  There exist points
$z'=z'(z,\z)$ and $z''=z''(z,\z)$
in the line segment $[z,\z]$ such that:
\begin{equation}\label{tedio} 
 \begin{aligned}
\phi_x(\z)&=\phi_x(z)+ \langle D \phi_x(z') ,\z-z\rangle+
 \abs{z}^{2m-1}O(\abs{\z-z}^2),
\\
  \phi_y(\z)&=\phi_y(z)+ \langle D \phi_y (z'') ,\z-z\rangle 
   +\abs{z}^{2m-1}O(\abs{\z-z}^2).
\end{aligned}
\end{equation} 
On the other hand, we have
\begin{equation*}
\begin{aligned}
&\abs{\z}^{2m}\eta=\abs{z}^{2m}y
+\abs{z}^{2m-2} \big \langle 
(2m xy ,\abs{z}^2+2my^2),\z-z\big \rangle  
+\abs{z}^{2m-1}O(\abs{\z - z}^2),
\\&
\abs{\z}^{2m}\xi= 
\abs{z}^{2m}x+
\abs{z}^{2m-2} \big\langle (\abs{z}^2+2m x^2, 2m xy ),\z-z\rangle 
+\abs{z}^{2m-1}O(\abs{\z - z}^2)
.
\end{aligned}
\end{equation*}

With these estimates, we have proved the following:

\begin{proposition}
Let $\phi\in C^\infty(\R^2)$ satisfy \eqref{terzina}.
For any $z,\z\in\R^2$ with $|\z|\leq 2 |z|$  we have 
\begin{equation*}
\begin{aligned}
 ZF(\z)=  
ZF(z)+ M(z,\z)(\z-z)+\abs{z}^{2m-1}O(\abs{\z-z}^2)
\end{aligned}
\end{equation*}
where $ZF=\begin{bmatrix}
           XF\\YF
          \end{bmatrix}
$, $\z-z=   
\begin{bmatrix}
          \xi-x\\\eta-y
          \end{bmatrix}
$
and $M=M(z,\z)$ is the $2\times2$ matrix
\begin{equation}\label{gervaso} 
 M  =\begin{bmatrix}
       \phi_{xx} (z')- 2m\abs{z}^{2m-2}xy  & \phi_{xy} (z')-
\abs{z}^{2m-2}(\abs{z} ^2+ 2m y^2)
       \\ 
       \phi_{xy}(z'')  +  \abs{z}^{2m-2}(\abs{z}  ^2+ 2m x^2)    &
\phi_{yy}(z'') + 2m\abs{z}^{2m-2}xy
      \end{bmatrix}.
\end{equation}
\end{proposition}

  If $\phi=0$, then $\det M= (1+2m)|z|^{4m}$. In this case,  the matrix
 $M$ is   nonsingular for all $z\neq 0$.
Example~\ref{esempietto} shows that, for  
some admissible functions,
nonsingularity may fail also at points $z\neq 0$.  However, we are able to show
 that the  matrix $M(z,\z)$  has always rank at least one and that  
it  satisfies the following quantitative nondegeneration property. This
property is needed to get an Ahlfors  lower bound in the noncharacteristic case, the Case
1c in next section.

\begin{proposition} \label{piove}
 Let $\phi\in C^\infty(\R^2)$ satisfy \eqref{terzina}. There exist 
constants $C_2>1$, $ \e_0, \e_2\in\left]0,1\right[  $   such that 
for all $z\neq 0$ there is a
unit
vector $u\in \mathbb{S}^1\subset\R^2$ such that for all $ r \in  \left]0,\e_0
\abs{z}\right[$   
we have 
\begin{equation}
 \label{proppo}
 \abs{M(z,\z)(\zeta -z)} \geq C_2^{-1 }\abs{z}^{2m} r 
\end{equation} 
for all $\z\in B_{\Eucl}\big(z+\frac r2 u  ,    \e_2 r   
\big)\subset B_\Eucl(z,r )\subset\R^2$. 
\end{proposition}

\begin{proof} 
Denote by $e_1, e_2$ the coordinate versors of $\R^2$. Then, letting
$M=M(z,\z)$, we have  
\begin{equation*}
\begin{aligned}
 \abs{M e_1}  + |M e_2| 
&  
\geq     
  \Big|     \phi_{xy}(z'')  +  \abs{z}^{2m-2}(\abs{z}^2+2m x^2)\Big| 
+\Big| \phi_{xy}(z') - \abs{z}^{2m-2}(\abs{z}^2+2m y^2)\Big|
\\
&
\geq   
(2m+2)\abs{z}^{2m} 
-\abs{\phi_{xy}(z')-\phi_{xy}(z'')}.
\end{aligned}
\end{equation*}
From \eqref{terzina} and 
$|z'-z''|\leq \abs{\z-z}\leq \e_0|z|$, we deduce that,  if $\e_0$ is conveniently small, then  we have the inequality
$
 \abs{\phi_{xy}(z')-\phi_{xy}(z'')} \leq |z|^{2m}
$. This implies that
$  \abs{M e_1}  + |M e_2|  \geq 2m \abs{z}^{2m}$.
Thus, given $z\neq 0$, at least one of the choices $u=e_1$ or $u=e_2$ ensures
that $\abs{M(z,\z)u}\geq    |z|^{2m}$ for all $\z $ such that $\abs{\z-z}<\e_0|z|$.
 Therefore,   for any $v$ with $|v|\leq 1$ and  $\e_2 >0 $  we have  
\begin{equation*}
|M(z,\z) (u+ \e_2 v  )|\geq  |z|^{2m}- \e_2|M(z,\z)|\geq  
   |z|^{2m}- \e_2 C|z|^{2m}
                    \end{equation*}
where $|M(z,\z)|$ denotes the operatorial norm, which under our assumptions
satisfies 
  $|M(z,\z)| \leq
C|z|^{2m}$. Thus, taking    $\e_2$  small enough we 
 get  a lower estimate with
$\frac12 |z|^{2m}$. 

The claim \eqref{proppo}  follows by multiplying the last inequality by
${r}/{2}$. 
\end{proof}

\section{Ahlfors property for entire admissible graphs}\label{DiMaio} 
In this section we prove Theorem~\ref{Ahl} 
in the case when the boundary of the domain is an entire admissible graph.
The case of a bounded domain is in Section~\ref{ultimissima}.
A discussion of the problem  in a translation  
invariant setting of step two is contained in \cite{CapognaGarofalo06}.

Let $\pi:\R^3\to\R^2$ be the projection $\pi(z,t) = z$
and denote by $\Sigma $ the graph of a function $\phi\in C^\infty(\R^2)$
satisfying \eqref{terzina}. By the area formula, the 
measure $\mu$ defined in \eqref{MES} satisfies, for any
$p\in\Sigma$ and $r>0$, 
\begin{equation} \label{wei}
  \mu(B(p,r)\cap\Sigma) = \int _{ \pi(B(p,r)\cap\Sigma)} |ZF(\z)| d\xi d\eta .
\end{equation}
The integration domain can be estimated using Lemma \ref{lascrivo}.
For any $z\in\R^2 $ and $r>0$ we define the ``disks''
\begin{equation} \label{dirro}
D(z,r) = \big\{
 \z\in\R^2: |\z -z| \leq r,\, 
 |\phi(\z)-\phi(z)+\mu_{\z,z}^{2m} \omega(z,\z) |
\leq \mu_{\z,z}^{2m} r^2\big\},
\end{equation}
where $\mu_{\z,z}=\max\{|z|,|\z|\}$. By Lemma \ref{lascrivo}, there exists a constant $C_0>0$ such that 
\begin{equation} 
\label{daxim}
 D(z,r/C_0)\subset \pi(B(p,r)\cap\Sigma) \subset D(z,C_0 r),\quad \text{for all $z\in\C$ and $r\in\left]0,+\infty\right[$}.
\end{equation}

The Lebesgue measure of the ball  $B(p,r)$, $p=(z,t) \in \R^3$, can be computed
using Corollary
\ref{scatola}.  
Let $ \e_0\in\left]0,1\right[ $ be the constant given by Proposition \ref{piove}. 
 From \eqref{giglio}--\eqref{giglio1} 
and 
\eqref{collo}--\eqref{collo1}, using Fubini-Tonelli Theorem we obtain
\begin{align*} 
 &|B(p,r)| \simeq |z|^{2m} r^4,\quad \text{if}\quad r\leq
\e_0\abs{z},
\\
\label{collo2}
 &|B(p,r)| \simeq r^{2m+4},\quad\,\,\,  \text{if}\quad r\geq
\e_0\abs{z}.
\end{align*}
 The equivalence constants depend on the parameter $\e_0$.

\subsection{Proof in the Case 1: $r\leq \e_0\abs{z}$.}
We claim that for any point $p=(z,\phi(z) )\in\Sigma$ and $r>0$ such that $r\leq
\e_0\abs{z}$
we have:
\begin{equation} \label{NULLO}
 \mu(B(p,r)\cap \Sigma)\simeq |z|^{2m} r^3.
\end{equation}

\noindent 
Observe that in Case 1 we have the obvious equivalence  
  $\frac 12|z|\leq|\z|\leq
\frac 32|z|$. 
Let $\beta>0$   be a parameter that will be fixed  
  after \eqref{terzovincolo}.   We distinguish two subcases:

\emph{Case 1c}: $\abs{z}^{2m}r \geq \beta\abs{ZF(z)} $. This
is  the characteristic case.

\emph{Case 1nc}:  $\abs{z}^{2m}r \leq\beta\abs{ZF(z)} $. This is the
non-characteristic case.

\medskip 

\noindent 
In the  Case  1c, points are in a quantitative way near the  characteristic set
of $\Sigma$, where $|ZF(z)|=0$.

 \step{Case 1c -- upper bound.}  We start from the elementary  
inclusion $ 
 \pi(B(p, r)\cap\Sigma )\subset \{\z\in\R^2  :\abs{\z-z}
 \leq   r\}$.
 Thus,  using the expansion \eqref{macchia2}  and the trivial estimate $|M
(\z-z)|  \leq C \abs{z}^{2m}\abs{\z-z} $ we obtain 
\begin{equation*}
\begin{aligned}
 \int_{\abs{\z-z}\leq r }\abs{ ZF(\z)}& d\xi d\eta
 = 
  \int_{\abs{\z-z}\leq r }\Big| ZF(z)+ M (\z-z)+\abs{z}^{2m-1}O(\abs{\z-z}^2)
\Big| d\xi d\eta 
 \\&
 \leq C\int_{\abs{\z-z}\leq r }\Big(\frac{1}{\beta}\abs{z}^{2m}r + C
\abs{z}^{2m}\abs{\z-z}
 +C \abs{z}^{2m-1}\abs{\z-z}^2
 \Big) d\xi d\eta
\\&
\leq C \Big(\frac{1}{\beta}+C\Big)\abs{z}^{2m}r^3.
\end{aligned}
  \end{equation*}
We also  used $r\leq \e_0|z|$  to estimate the
third term.

\step{Case 1c  -- lower bound.}
We claim that there exist constants $\e_1>0$ and $\b>0$ such that 
\begin{equation} \label{ghio}
 \{ \z  \in\R^ 2 :|\z-z| \leq \e_1 r/C_0\} \subset
\pi(B(p,r)\cap\Sigma). 
\end{equation} 
The constant $C_0$ is the one given by Lemma \ref{lascrivo}.

In view of the expansion \eqref{macchia2}, the set $D(z,r/C_0)$ introduced in \eqref{dirro} satisfies 
\[
 D(z,r/C_0) = \big \{ \z\in\R^2: |\z-z|\leq r/C_0,\, 
| \scalar{ZF(z)}{\z -z}
 +\mass^{2m}O(\abs{\z-z}^2)   |
 \leq C_0^{-2}  \mass^{2m}r^2 \big \},
\]
where, 
 for some absolute constant $C_1$, we have
  $|O(\abs{\z-z}^2)|\leq C_1\abs{\z-z}^2 \leq  C_1 C_0^{-2} \e_1^2 r^2$,
provided that $|\z-z| \leq \e_1 r /C_0$.
If $\e_1$ satisfies 
\begin{equation}\label{primovincolo} 
 C_1\e_1^2\leq \frac 12,
\end{equation} 
  then we have the inclusion
\begin{equation} \label{gaio}
  \big \{ \z\in\R^2: |\z-z|\leq \e_1  r/C_0,\, 
| \scalar{ZF(z)}{\z -z}
    |
 \leq \frac 12  C_0^{-2} \mass^{2m}r^2\big \}\subset  D(z,r/C_0) .
\end{equation}
By the Case 1c, we have
\begin{equation*}
 \abs{\scalar{ZF(z)}{\z-z}} \leq \frac{1}{\beta}\abs{z}^{2m}r \abs{\z-z}
 \leq \frac{1}{\b} \e_1  C_0^{-1}\mass^{2m} r^2.
\end{equation*}
Thus, if  $\e_1$  and $\beta$ are such that 
\begin{equation}\label{secondovincolo} 
 \frac{\e_1 C_0^{-1}}{\beta}\leq \frac{C_0^{-2}}{2},
\end{equation} 
the inclusion  \eqref{ghio} holds, as we claimed.

 Next we use  
Proposition~\ref{piove} to  estimate  $|ZF(\z)|$ from below at all points 
$\z\in B_{\Eucl}\big(z+\varrho u/2 ,  \e_2  \varrho   
\big)$,
where  
$\varrho=\e_1 C_0^{-1}r$ and $u\in\mathbb S^1$ is such that \eqref{proppo}
holds. Namely, we have 
\[
 |M(\z-z)|\geq C_2^{-1}|z|^{2m}\e_1 C_0^{-1}r,
\]
where $M = M(z,\z)$ is the matrix \eqref{gervaso}. 
For such a point $\z$ we   have
\begin{equation}
\label{origlio} 
\begin{aligned}
 |ZF(\z) | & =| 
ZF(z)+ M(\z-z)+\abs{z}^{2m-1}O(\abs{\z-z}^2) |
\\
& \geq C_2^{-1}|z|^{2m}\e_1 C_0^{-1}r 
 -|ZF(z)|-|z|^{2m-1}O(\abs{\z-z}^2
 )
 \\&\geq    C_2^{-1}|z|^{2m} \e_1 C_0^{-1}r  - \frac{1}{\b}|z|^{2m }r -
C_1|z|^{2m-1}
 \e_1^2  r^2
 \\&
 \geq
  C_2^{-1}|z|^{2m} \e_1 C_0^{-1}r  - \frac{1}{\b}|z|^{2m }r - C_1|z|^{2m }
 \e_1^2  r 
 \geq \frac 12 C_2^{-1}|z|^{2m} \e_1 C_0^{-1}r,
\end{aligned}
\end{equation}
provided that 
\begin{equation}
 \label{terzovincolo}
 \frac{1}{\beta}\leq \frac 14 C_2^{-1}  \e_1 C_0^{-1} \qquad\text{and }
 \quad 
 C_1 
 \e_1^2   \leq  \frac 14 C_2^{-1}  \e_1 C_0^{-1}.
\end{equation} 
We
may choose $\e_1>0$ such  that
the inequalities in the right-hand side of \eqref{terzovincolo} and in
\eqref{primovincolo} both hold.
Then  we fix $\beta>0$  such that  the  
inequalities  in the   left-hand side of
\eqref{terzovincolo} and in \eqref{secondovincolo} both hold.

By \eqref{ghio} and \eqref{origlio}, we finally obtain
\[
\int _{ \pi(B(p,r)\cap\Sigma)} |ZF(\z)| d\xi d\eta \geq  C_3^{-1} |z|^{2m} r^3,
\] 
where, after fixing $\e_1$ and $\beta$, the constant $C_3$ is absolute.

\step{Case 1nc -- upper bound.}   In order to evaluate from above the
integral in \eqref{wei},  we start from the    estimates
$
 \abs{M (\z-z)}\leq C \abs{z}^{2m}\abs{\z-z}\leq C\abs{z}^{2m}r\leq
C\beta|ZF(z)|
$,
where $C$ is an absolute constant. Here we used the fact that $|\z-z|<r$ for all points $\z$ in the integration set. Therefore, the weight in the integral
\eqref{wei} satisfies
\begin{equation} \label{silla}
\begin{aligned}
|ZF(\z)|= \big|ZF(z)+ M (\z-z)
    +\abs{z}^{2m-1}O(\abs{\z-z}^2)  \big|
     \leq C (1+ \beta )|ZF(z)| .
\end{aligned}
\end{equation}
In order to get the required estimate,  the obvious inclusion
$\pi(B(p,r)\cap\Sigma)\subset B_{\textrm{Eucl}}(x,r)$ does not suffice. We need
  the stronger condition   \eqref{gallo}, which tells that  for some
absolute constant $C_0>0$ we have $ 
\pi(B(p,r)\cap\Sigma) \subset
D(z,C_0 r)$.
Using the definition \eqref{dirro} of $D(z,C_0r)$, the expansion \eqref{macchia2} and also using $|z|\simeq |\z|$, that follows from
Case 1, we
discover that 
 \begin{equation*}
\Big| \Big\langle\frac{ZF (z)}{\abs{ZF (z)}},\z-z \Big\rangle
\Big|\leq C\frac{\abs{z}^{2m}}{|ZF(z)|} r^2,\qquad \text{for all $\z\in D(z, C_0r)$} .
\end{equation*} 
This tells that  the  projection of the  set $D(z, C_0r)$ along the unit direction $\frac{ZF(z)}{|ZF(z)|}$ has size $\frac{|z|^{2m}}{|ZF(z)|}r^2$. Therefore  the Lebesgue measure of $\pi(B(p,r)\cap\Sigma)$
satisfies the inequalities 
\begin{equation} \label{mario}
 |\pi(B(p,r)\cap\Sigma)| \leq |D(z,C_0 r)| \leq C\frac{\abs{z}^{2m}}{|ZF(z)|}
r^3,
\end{equation}
for an absolute constant $C>0$. 
Ultimately, from \eqref{silla} and \eqref{mario}, we obtain the upper-bound:
\[
  \mu(B(p,r)\cap \Sigma) = \int _{\pi(B(p,r)\cap\Sigma)} |ZF(\z)|d\xi d\eta 
  \leq  C (1+\beta) \abs{z}^{2m}  r^3.
\]

\step{Case 1nc -- lower bound.} 
Observe first that  \eqref{gaio} holds also in this case. Then, under our choice
of $\e_1$ and $\b$, we have
\begin{equation*}
\begin{aligned} 
\Big\{\z\in\R^2 : \abs{\z-z}\leq \e_1 C_0^{-1}r \;
 \text{ and }\abs{ \scalar{ZF(z)}{\z-z}} \leq  \frac 12
\mass^{2m}C_0^{-2}r^2\Big \}\subset 
\pi(   B(p, r)\cap\Sigma)  .
\end{aligned}
\end{equation*}
 Let $  \e_3\leq \e_1$ be a small positive constant to be fixed
below. Then we have
\begin{equation*}
\begin{aligned} 
\Big\{\z\in\R^2 : \abs{\z-z}\leq \e_3 C_0^{-1}r \;
 \text{ and }\abs{ \scalar{ZF(z)}{\z-z}} \leq  
\abs{z} ^{2m}C^{-1}r^2\Big
\}\subset 
\pi(   B(p, r)\cap\Sigma)  ,
\end{aligned}
\end{equation*}
The  set in the left-hand side has size $\frac{|z|^{2m}}{|ZF(z)|} r^2$
along the unit direction $ZF(z)/|ZF(z)|$. Then 
we have the estimate from below for the Lebesgue measure of the integration set.
\begin{equation} \label{AA}
 |\pi(   B(p, r)\cap\Sigma)| \geq C \frac{\abs{z}^{2m} r^ 3}{|ZF(z)| }.
\end{equation}

To conclude the proof, we get a lower estimate for the function in the integral. 
\begin{equation}
 \label{BB} 
\begin{aligned}
|ZF(\z)|& =  \big|ZF(z)+ M (\z-z)+\abs{z}^{2m-1}O(\abs{\z-z}^2)  \big|
   \geq |ZF(z)|- C_4  | z |^{2m} \abs{\z-z} 
 \\&
 \geq |ZF(z)|-C_4|z|^{2m}\e_3 C_0^{-1}r  
 \\&
 \geq |ZF(z)|-
 C_4 \e_3 C_0^{-1}\beta |ZF(z)|
 \geq \frac 12|ZF(z)|, 
\end{aligned}
\end{equation}
provided that $\e_3$ is so small that 
$C_4\beta\e_3 C_0^{-1}\leq \frac 12$. 
  The lower bound follows from~\eqref{AA} and~\eqref{BB}. This ends the proof of~\eqref{NULLO}.

\subsection{Proof in the Case 2: $r\geq
\e_0\abs{z} $. }
We claim that for any point $p=(z,\phi(z) )\in\Sigma$ and $r>0$ such that $r\geq
\e_0\abs{z}$
we have: 
\begin{equation} \label{NULLO2}
 \mu(B(p,r)\cap \Sigma)\simeq  r^{2m+3}.
\end{equation}

We can without loss of generality assume that $\phi(0)=0$.
By Lemma \ref{lascrivo}  and  $\abs{\phi(z)}\leq
C|z|^{2m+2}$, there     exists a constant $C_5>1$ such that  
\begin{equation}\label{gotti} 
 C_5^{-1}|z|\leq d((0,0) , (z,\phi(z)))\leq C_5 |z|.
\end{equation}

\step{Case 2 -- upper bound.}   
For  $z\in\R^2$ and  $r\geq \e_0 |z|$, we have the inclusions
\[
  B((z,\phi(z)), r) \subset B((0,0), C_5|z| + r)\subset B
\big((0,0),(\e_0^{-1}+C_5)r\big).
\]
Thus, we obtain 
\begin{equation*} \int_{\pi( B(p, r)\cap\Sigma)
 }|ZF(\z)|d\xi d\eta  \leq \int_{|\z|
   \leq C r} C|\z|^{2m+1}d\xi d\eta 
\leq Cr^{2m+3}. 
\end{equation*}

To check the lower bound, we distinguish two cases.
 
\step{Case 2a:  $\e_0|z|\leq r
\leq 2C_5 |z| $.}  It suffices to start from inclusion $B((z,\phi(z)), r)\supset
 B((z,\phi(z)), \e_0 |z|)$. Then the perimeter measure of the smaller ball can
be estimated as in  Case 1.  To conclude observe that $|z|\simeq r$.


\step{Case 2b:  $2C_5 |z|\leq r<\infty$.}  
  In such case we have  
\begin{equation}
\label{duosp}  
B\big( 0,  r/({2C_5^2 })\big)\cap\Sigma \subset  B(p, r)\cap\Sigma .
\end{equation} 
Indeed, for any $q=(\z,\phi(\z))\in B(0,{r}/({2C_5^2 })
)$   we have $|\z|\leq r/(2C_5)$ and 
\begin{equation*}
\begin{aligned}
 d(p,q)\leq d(0, p)+ d(0, q) \leq C_5|z|+C_5|\z|
\leq r,
\end{aligned}
\end{equation*}
as  claimed.  

By \eqref{duosp}, it is enough to prove the lower-bound estimate in the case
$z=0$. 
For $\varrho>0$, we calculate by Stokes'  theorem  the following integral on the
curve 
$\gamma_\varrho
(s)=\varrho e^{-is}$,  with $s\in[0,2\pi]$, 
 \begin{equation*}
\begin{aligned}
  \int_{\gamma_\varrho }(XF d\xi + YF d\eta) &= \int_{|\z|<\rho} \Big( -
  \p_\eta 
  (\p_\xi \phi(\z)- |\z|^{2m}\eta )+ \p_{\xi}( \p_{\eta}\phi(\z)+ |\z|^{2m}\xi) 
  \Big) d\xi d\eta 
  \\& = (2+2m)\int_{|\z|<\r}|\z|^{2m}d\xi d\eta = C^{-1} \r^{2m+2},
\end{aligned}
 \end{equation*}
that implies
$  
  \int_{|\z| = \r }|ZF(\z)| d\mathcal{H}^1(\z) \geq C^{-1} \r^{2m+2}.
$
Thus we get the  estimate
\begin{equation*}
 \int_{|\z|<r}|ZF(\z)|d\xi d\eta =\int_0^r\int_{|\z|=\r} |ZF(\z)| 
 d\mathcal{H}^1(\z)  d\rho \geq C^{-1}r^{2m+3}.
\end{equation*}
This inequality ends the proof of \eqref{NULLO2} and thus of Theorem \ref{Ahl} 
in the case of entire
admissible graphs.

\section{Cone  property for  admissible entire epigraphs}
\label{CONO}

We recall the definition of a John domain, specialized to the metric space
$(\R^3,d)$.

 \begin{definition} \label{deffo}
 A bounded open   set
$\Omega\subset\R^3$
is a $\lambda$-John domain, with $\lambda>0$, if  there exists a point 
$p_0\in\Omega$ such that for all $p\in\Omega$ there is  a continuous curve
$\gamma:[0,1]\to \Omega$ such that $\gamma(0)=p$, $\gamma(1)= p_0$,
and 
\begin{equation} \label{jolly}
   B\big( \gamma(t), \lambda \diam(\gamma|_{[0,t]})\big)\subset\Omega
\quad\text{for all $t\in  \left]0,1\right[$.}
\end{equation}
A curve $\gamma$  satisfying \eqref{jolly}
is called a \emph{John curve in $\Omega$ with parameter $\lambda>0$}.
 \end{definition}

In our definition, the curve $\gamma$ is  not required to be rectifiable. By
the results
of \cite{MartioSarvas79}, Definition \ref{deffo} is equivalent to the more
standard one with rectifiable curves and  with diameters replaced by lengths.  
John domains are also known as  domains with the twisted interior
cone property.

In this section, we consider an  unbounded  domain of the epigraph type $\Omega
= \mathrm{epi}(\phi) = \{ (z,t) \in\R^3 : t>\phi(z)\}$, where $\phi\in
C^\infty(\R^2)$ is an $m$-admissible function,
and we construct a nontrivial John curve starting from any point
$p=(z,t)\in\Omega$.
The case of a bounded domain is discussed in Section \ref{ultimissima}.

For any point $z\in \R^2$ with $|ZF(z)|\neq 0$ there exists  a unit
vector $(u,v)\in \mathbb S^1 \subset \R^2$ 
such that 
\begin{equation} \label{MAX}
  -(uX+vY)F(z)>\frac 12 |ZF(z)|,
\end{equation}
where  $F(z,t)= \phi(z)-t$.
Our John curve starting from $(z,t) \in \mathrm{epi}(\phi)$ is the integral
curve of $-(uX+vY)$ for times  $s \in [0,\bar s]$, where the time $\bar s = \bar
s(z)$ is 
\begin{equation}\label{tempogiusto} 
  \bar  s =\e_0\frac{|ZF(z)|}{\abs{z}^{2m}}, 
\end{equation} 
and $\e_0>0$ is a suitable    constant. For $s\geq \bar s $, the John 
curve  is an integral curve of $\partial / \partial t$. This piece of curve
is nonrectifiable.
When  $(z,\phi(z))\in\mathrm{gr}(\phi)  $ is a characteristic point
of $\mathrm{gr}(\phi)$,
i.e.~$|ZF(z)|=0$,
we have  $\bar s  = 0$ and the first piece of the curve 
disappears.

By the rotational invariance \eqref{rotoliamo} of the metric   $d$, in
\eqref{MAX} we can assume
that $u=1$ and $v=0$.

\begin{theorem} \label{giovannino} Let $\phi\in C^\infty(\R^2)$ be
a function  satisfying \eqref{terzina}. There exist constants $\e_0>0$ and 
$\la>0$ such that
for any $z\in \R^2$ with 
 \begin{equation}\label{massimale} 
-XF(z) 
>\frac 12 |ZF(z)|, 
\end{equation}  
the curve $\gamma:[0,\infty)\to \R^3$  
 \begin{equation*}
  \gamma(s)= 
  \left\{
\begin{array}{ll}
  e^{sX}(z,t) 
, & \text{if $\quad 0\leq s\leq \bar s =\e_0\frac{|ZF(z)|}{\abs{z}^{2m}}$}
\\
   \gamma(\bar s) + (0, s-\bar s),
&
   \text{if $\quad \bar s \leq s<\infty$}
\end{array}
  \right.
 \end{equation*} 
is a John curve in $ \mathrm{epi}(\phi)$ with parameter $\la$ starting
from $(z,t)\in \mathrm{epi}(\phi)$.
\end{theorem}

Let  $\lambda>0$ be the 
 parameter of our John curves. 
 In the proof of the theorem and in the following sections, 
we denote  by~$\s_\la$ any constant  of the form $C\lambda^\b$, where $C$ is an
absolute constant and $\b >0$ is a positive power.

\begin{proof}
Without loss of generality, we prove the claim for $t= \phi(z)$, i.e.,
we construct a John curve in the epigraph of $\phi$ starting from a 
boundary point. In the following we let $z_s = z+ se_1$. 
When $s\in [0,\bar s]$, an explicit formula for $\gamma(s)$ is 
 \begin{equation*} 
  \gamma(s)= e^{sX}(z,\phi(z) ) 
           = \Bigl( z_s, \phi(z)+  y\int_0^s
\abs{z_\r}^{2m}d\r \Bigr).
 \end{equation*} 
The definition in \eqref{tempogiusto} for $\bar s $ implies that 
$\bar  s \leq\e_0 C\abs{z}$, where $C>0$ is the constant appearing in
\eqref{wow}.
We can choose  $\e_0>0$ such that   $\e_0 C <\frac 12$.
Then, for any $s\in [0,\bar s ]$ we have
 \begin{equation}
\label{lazzeta} 
  \frac{\abs{z}}{2}\leq   \abs{z_s} \leq\frac 32|z|.
 \end{equation} 
Further conditions on $\e_0$ will be required below.

\step{Step 1.} We claim that  for any sufficiently small $\la>0$
we have $B(\gamma(s),
\lambda s)\subset \mathrm{epi}(\phi)$ for all    $0< s \leq  \bar s  $.

From $s\leq \bar s  \leq \frac 12 |z|$ and Corollary \ref{scatola}, we have the
inclusion  $ B(\gamma(s), \lambda s)\subset \Box_I(\gamma(s), \sigma_\la  s)$
for any  $0<\la\leq 1$, where $\sigma_\la = b_1\la$ and $b_1$ is the constant
given by Corollary \ref{scatola}. So our claim is
implied by $\Box_I(\gamma(s), \sigma_\la  s)\subset \mathrm{epi}(\phi)$.
We have $p \in  \Box_I(\gamma(s), \sigma_\la  s)$ if and only if 
 \begin{equation*}
 p =  \Big (z_s+v , \phi(z)+  y\int_0^s |z_\r |^{2m} d\r  +  
 |z_s |^{2m } \big(u_3 + \omega(v,z_s) \big)  \Big)  ,
 \end{equation*}
with $\norm{u}_{1,1,2}\leq \s_\la s$ and $u = (u_1,u_2,u_3) = (v,u_3) $.
Then the claim in Step 1 is implied by the inequality 
 \begin{equation}\label{lot2} 
\begin{aligned}
 \phi( z_s+ v)- \phi(z)< y\int_0^s |z_\r |^{2m} d\r +
 |z_s|^{2m } \big(u_3 + \omega(v,z_s)  \big)=: H,
\end{aligned}
\end{equation}
for  $\norm{u}_{1,1,2}<\s_\la s$ and $s\leq \bar s $. 
The left-hand side of \eqref{lot2} can be expanded using~\eqref{marrone}:
\[
 \phi( z_s+ v)- \phi(z) =\langle ZF(z), v_s\rangle + \abs{z}^{2m}
\omega(v_s,z) +  \mu_{z,z_s} ^{2m} O(|v_s|^2 ).
\]
By \eqref{massimale},  $|z_s|\leq C|z|$, and    $u_1+s\geq 0$ we get
\[
 \phi( z_s+ v)- \phi(z) \leq |ZF(z)| \Big( -\frac 12  (u_1+s) +   |u_2|\Big)  +
\abs{z}^{2m}
\omega(v_s,z) + C|z|^{2m} |v_s|^2.
\]
Since $|u_2|\leq \s_\la s$ and $|u_1|\leq \s_\la s$, we have $|v_s |\leq C s$
and for small $\lambda>0$, we get 
\begin{equation} 
\label{pico}
 \phi( z_s+ v)- \phi(z) \leq -\frac s 4 |ZF(z)|  +
\abs{z}^{2m}
(\omega(v,z)+sy+C_0 s^ 2)   .
\end{equation}

Now we estimate the quantity $H$ in the right-hand side of  \eqref{lot2}.
Since $\r\leq  s\leq \bar s \leq \frac 12|z|$, and $u_3\geq  - \sigma_\lambda
 s^2$  we  have the elementary inequalities
\begin{equation}\label{cuccurullo} 
 \begin{aligned}
 & 
\big||z_\r |^{2m} -|z|^{2m}\big| \leq C|z|^{2m-1}\r,
\\
&
|z_s |^{2m}u_3  \geq -C|z|^{2m}\s_\la   s^2, \quad \text{and }
\\ 
&
|z_s |^{2m}\omega(v,z_s)
\geq |z|^{2m}\omega(v,z_s)-C|z|^{2m} s^2.
\end{aligned}
\end{equation}
Thus, we obtain 
\begin{equation}
 \label{paco}
\begin{aligned}
 H & \geq   |z|^{2m} (\omega(v,z)+y s  -C_1   s^2)  
\end{aligned}
\end{equation}
for a suitable absolute constant $C_1$.

By \eqref{pico} and \eqref{paco},  the claim \eqref{lot2} is then implied by the
inequality
 \begin{equation*}
 \begin{aligned}
(C_0+C_1)|z|^{2m} s<\frac{1}{4}|ZF(z)| ,
\end{aligned}
\end{equation*} 
that 
holds for all $s\leq \bar s = 
\e_0 \frac{|ZF(z)|}{|z|^{2m}}$ as soon as $\e_0\leq\frac{1}{4(C_0+C_1)}$.

\step{Step 2. } For $s\geq  \bar s $, the curve  $\gamma$ is defined
by the formula 
 \begin{equation*}
 \gamma(s)=\Bigl(  z_{\bar s} , \phi(z)+y\int_0^{\bar s }\abs{z_\r }^{2m}d\r +
  s-\bar s  \Bigr).
\end{equation*}   
In the following we let   $s-\bar s =\t^{2m+2}$. As a
function of $\tau$, the
diameter of $\gamma$ restricted to $[0,s]$ satisfies  
\begin{equation*}
\Delta_\tau = \begin{aligned}
\operatorname{diam}(\gamma|_{[0, \bar s +\t^{2m+2}]}) & 
\leq C\Big(  \bar s  +\min\Big\{\t, \frac{\t^{m+1}}{\abs{z}^m}\Big\} \Big)
=: \bar s +m_\t .
\end{aligned}
\end{equation*}
The proof of  Theorem~\ref{eqq} shows in fact that the inequality above is a
global equivalence for $\ol s,\tau\in\left[0,+\infty\right[$.

We claim that  for any sufficiently small $\la>0$
we have $B(\gamma(s),
\lambda \Delta_\tau )\subset \mathrm{epi}(\phi)$ for all    $\t\geq 0$. 
By Corollary \ref{scatola}, this claim is equivalent to
\begin{align}
 &\Box_I(\gamma(s), \s_\lambda \Delta_\t)\subset\mathrm{epi}(\phi) \label{uu}
\quad  
\text { and }
\quad \Box_J(\gamma(s), \s_\lambda \Delta_\t)\subset\mathrm{epi}(\phi)  .
\end{align}

We prove the inclusion in the left-hand side of \eqref{uu}. 
We have  $p\in \Box_I(\gamma(s), \s_\lambda \Delta_\t)$ 
if and only if   
\begin{equation}
\label{cok} 
 p = \Bigl( z_{\bar s} +v , 
 \phi(z)+
 y\int_0^{\bar s }\abs{z_\r }^{2m}d\r
 +\abs{z_{\bar s} }^{2m}\big(u_3+ \omega(v, z_{\bar s}) \big) +\t^{2m+2}
 \Bigr),
\end{equation}
with $\norm{u}_{1,1,2}\leq \s_\la \Delta_\tau $ and $u = (u_1,u_2,u_3) = (v,u_3)
$. The point $p$ belongs to $\mathrm{epi}(\phi)$ if 
\begin{equation*}
\begin{aligned}
L : =  \phi(z_{\bar s} +v )
 - \phi(z) <   y\int_0^{\bar  s}\abs{z_\r}^{2m}d\r 
 +\abs{z_{\bar s} }^{2m}\big(u_3+\omega(v, z_{\bar s})\big)+\t^{2m+2} = : R
 .
\end{aligned}
\end{equation*}

First we get an upper bound for $L$.  Observe first 
that assumption~\eqref{massimale}
  and   the inequalities
    $\bar s\leq\frac 12 \abs{z}$
  and $\abs{v}\leq
\s_\la(\bar s+m_\t)$ give    
\begin{equation*}
\begin{aligned}
 \langle ZF(z), v_{\bar s} \rangle 
 \leq |ZF(z)| \Big( -\frac 14   \bar s+\s_\la m_\t\Big) .
\end{aligned}
\end{equation*}
 Thus, using formula  \eqref{marrone}
we obtain 
\begin{equation}\label{giostraio} 
\begin{aligned}
 L  & = \phi(z+ v_{\bar s} ) - \phi(z) = \langle ZF(z), v_{\bar s} \rangle 
+ \abs{z}^{2m}  \omega (v_{\bar s}, z)  
+ \max\big\{ |z|^{2m},|v_{\bar s}  |^{2m}\big\} O\big(\abs{v_{\bar s} }^2\big)
\\&
\leq |ZF(z)| \Big( -\frac 14   \bar s+\s_\la m_\t\Big)  
+ \abs{z}^{2m}  \omega (v_{\bar s}, z)   +C\big(|z|^{2m}+\s_\la
m_\t^{2m}\big)
(\bar s^2+\s_\la^2 m_\t^2) 
\\&
\leq  |ZF(z)| \Big( -\frac 14   \bar s+\s_\la m_\t\Big)  
+ \abs{z}^{2m}   \Big( \omega (v , z)  +\bar s y
+C_0 \bar s^2 + \s_\la  m_\t^2\Big) +\s_\la m_\t^{2m+2} .
\end{aligned}
\end{equation} 

\noindent 
We compute a lower bound for the right-hand side $R$.
Using  \eqref{cuccurullo} we get 
\begin{equation}\label{caccona} 
\begin{aligned}
 R & =  y\int_0^{\bar s}\abs{z_\r}^{2m}d\r+|z_{\bar s}|^{2m }
 \big( u_3+\omega(v, z_{\bar s} )  \big) +\t^{2m+2} 
\\&
\geq y|z|^{2m}\bar s -  C |z|^{2m}\bar s^2 +\t^{2m+2} 
 -\sigma_\lambda |z|^{2m}  m_\t^2
+|z|^{2m} (\omega(v,z )  - u_2\bar
s )
\\
&
\geq 
\t^{2m+2} +|z|^{2m} \Big( \omega(v,z ) + \bar s y-  C_1 \bar s^2 
-\sigma_\lambda \bar   m_\tau^2 \Big  ).
 \end{aligned}
\end{equation} 
Then, the  inequality  $L< R$ follows  from  
\begin{equation}
 \label{gagliardo}
 \begin{aligned}
 \s_\la m_\t  |ZF(z)| 
+ \abs{z}^{2m}   \Big( C_2 \bar s^2 + \s_\la  m_\t^2 \Big) +\s_\la m_\t^{2m+2} 
<\frac 14 \bar s |ZF(z)|+
\t^{2m+2}   .\end{aligned}
\end{equation} 
To  prove \eqref{gagliardo}, we  start from the second term.
By the definition of $\ol s$, we have 
\begin{equation*}
C_2 \bar s^2 |z|^{2m} =
   \e_0 C_2 |ZF(z)|  \bar s\leq \frac 14 |ZF(z)| \bar s ,
\end{equation*}
 as soon as $\e_0$ satisfies  $\e_0 C_2<1/4$.
This is the last time we modify the choice of $\e_0$.

Next we look at the first term. Observe that
\begin{equation*}
\begin{aligned}
\s_\la  |ZF(z)|m_\t = \s_\la
|ZF(z)|\min\Big\{\t,\frac{\t^{m+1}}{|z|^m}\Big\}\leq 
\s_\la |ZF(z)|\frac{\t^{m+1}}{|z|^m}
 \leq \s_\la \Big( \frac{|ZF(z)|^2}{|z|^{2m}}+\t^{2m+2}\Big).
\end{aligned}
\end{equation*}
Then, since $\frac 14 |ZF(z)| \bar s+\t^{2m+2}
=\frac{\e_0}{4}\frac{|ZF(z)|^2}{|z|^{2m}}+\t^{2m+2}$, we can finish
the estimate  as soon as $\s_\la$ is small with respect to absolute 
constants (which include $\e_0$, now). 

The estimate of the third term is easy:
\begin{equation*}
 \s_\la \abs{z}^{2m} m_\t^2 =\s_\la |z|^{2m} \Big( \min\Big\{\t,\frac{\t^{m+1}}{|z|^m}\big\}\Big)^2
 \leq \s_\la \t^{2m+2},
\end{equation*}
which is correctly estimated, provided that $\s_\la$ is small enough.
Finally,  we have
\begin{equation*}
 \s_\la m_\t^{2m+2}=\s_\la\Big( \min\Big\{\t,\frac{\t^{m+1}}{|z|^m}\Big\}\Big)^{2m+2}
 \leq \s_\la\t^{2m+2},
\end{equation*}
which again satisfies the required estimate.

To conclude the proof, we have to check the inclusion  in the right-hand
side of \eqref{uu}.
In this case the box $\Box_J(\gamma(s), \s_\la(\bar s+m_\t)) $ is made  of
points of the form
\begin{equation*} 
 \Bigl(  z_{\bar s}+v, 
 \phi(z)+
 y\int_0^{\bar s}\abs{z_\r }^{2m}d\r +\t^{2m+2}
+u_3 +\abs{z_{\bar s}}^{2m} \omega (v,z_{\bar s})   
 \Bigr).
\end{equation*}
The unique difference with   \eqref{cok} is that the term $u_3$ replaces the
term
$\abs{z_{\bar s }}^{2m}
u_3$, and now $|u_3|\leq\s_\la(\bar s+m_\t)^{2m+2}$.

The estimate from above for   $L $ remains unchanged, because it does
not involve $u_3$. In the estimate from below for $R$, we need the following evaluation for
the term $u_3$:
 \begin{equation*}
  u_3\geq -\s_\la(\bar s+m_\t)^{2m+2}\geq -\s_\la |z|^{2m}\bar s^2-\s_\la
m_\t^{2m+2}.
 \end{equation*}
Therefore,  the inequality  \eqref{gagliardo} remains unchanged and the proof 
can be  concluded arguing as in the previous case.  
\end{proof}

\section{Uniform property of entire  admissible 
epigraphs}\label{disconosco}

We recall the definition of a uniform domain, specialized to the metric
space
$(\R^3,d)$.

 \begin{definition}
\label{UNI}
 An open set  $\Omega\subset \R^3 $ is a uniform domain if there exist
$\e >0$ and $\delta>0$  with the following property. For any
pair of points $x,y\in\Omega$ there is a continuous curve $\gamma:[0,1]\to\Omega
$ such that
$\gamma(0)=x$, $\gamma(1)=y$, 
  \begin{equation}
   \label{penna} 
\diam(\gamma)  \leq \delta^{-1}d(x,y),
  \end{equation}
and, letting $\Delta_t = \min\{\diam(\gamma|_{[0,t]}),
\diam(\gamma|_{[ t,1]})\}$, for any $t\in [0,1]$ we have
  \begin{equation}
\label{sailor}   
B  ( \gamma(t),\e \Delta_t  )\subset \Omega .
  \end{equation}
 
 \end{definition}

Uniform domains are also known as  $(\e,\delta)$-domains.
As for John domains, the curves in our definition  are not required to be 
rectifiable. By the results
of \cite{MartioSarvas79}, this   is equivalent to the more standard definition
which requires rectifiability.

We consider an  unbounded  domain of the epigraph type $\Omega
= \mathrm{epi}(\phi) = \{ (z,t) \in\R^3 : t>\phi(z)\}$, where $\phi\in
C^\infty(\R^2)$ is an $m$-admissible function. 
For any pair of points $p,q \in \Omega$,  we construct a  curve connecting them 
and satisfying the conditions \eqref{penna} and \eqref{sailor} with uniform constants $\delta$ and $\e$.
The case of a bounded domain is discussed in Section \ref{ultimissima}.

\begin{theorem}
 \label{buonino}
Let $\phi\in C^\infty(\R^2)$ be a function satisfying  
\eqref{terzina}. Then, the epigraph $\Omega= \mathrm{epi}(\phi)$ 
is a uniform domain.
\end{theorem}

\begin{proof}
Let $p=(z,\phi(z)+b)$ and $q=(\z,\phi(\z)+\b)$, with $b,\beta > 0$, be points
in the epigraph of $\phi$. We can  without loss of generality assume that 
 \begin{equation}\label{massimino} 
\frac{|ZF(z)|}{\abs{z}^{2m}}  =\max                            
\Big\{ \frac{|ZF(z)|}{\abs{z}^{2m} }, \frac{|ZF(\z)|}{\abs{\z}^{2m}} \Big\} \geq
0.
\end{equation}
The  maximum can be $0$,  even for arbitrarily close points. This happens for
instance in Example \ref{esempietto}.   
 By assumption \eqref{terzina}   we can define continuously $\frac{|ZF(z)|}{|z|^{2m}}=0$ for $z=0$.

 Let $\mu>0$ be a parameter that will be fixed along the proof.
We distinguish two cases:
 \begin{equation*}
  \begin{aligned}
d(p,q)< \mu \max\Big\{ \frac{|ZF(z)|}{\abs{z}^{2m}}, \frac{|ZF(\z)|}{\abs{\z}^{2m}}
\Big\}\qquad&\text{(Case A);}
\\ 
d(p,q)\geq  \mu \max\Big\{ \frac{|ZF(z)|}{\abs{z}^{2m}},
\frac{|ZF(\z)|}{\abs{\z}^{2m}} \Big\}\qquad&\text{(Case B).}
\end{aligned}
 \end{equation*} 
 Roughly speaking, the maximum appearing in the right-hand side describes  quantitatively ``how much'' the involved points are close to the characteristic set.  In case~A, where $d(p,q)$ is much smaller than such maximum, the first pieces of the John curves from $p$ and $q$ are ``parallel'' and the curve realizing the uniform condition is constructed using the first pieces of $\gamma_z$ and $\gamma_\z$. In Case B, where the distance among  $p$ and $q$ is large,  the curve realizing the uniform condition is constructed using both pieces of the John curves starting from $p $ and $q$.
\color{black}

 We can without loss of generality assume that 
 \eqref{massimale}  holds at the point $z$, i.e.:  $-XF(z)=\abs{XF(z)}>\frac 12
|ZF(z)|$. Then, if we denote by $\e_0$ and 
$\lambda>0$   the  parameters  fixed in Section~\ref{CONO}, we know that 
the curve 
 \begin{equation*}
  \gamma_z(s)= 
  \left\{
\begin{aligned}
& \Bigl( z_s,  \phi(z) +b+ y\int_0^{s }\abs{z_\r}^{2m}d\r
\Bigr),\quad\text{if
$ s\leq \bar s = \e_0 \frac{|ZF(z)|}{\abs{z}^{2m}} $}
\\& 
\Bigl( z_{\bar s},  \phi(z) +b+ y\int_0^{\bar s }\abs{z_\r}^{2m}d\r
+s-\bar s \Bigr),\quad\text{if
$ s\geq \bar s$},
\end{aligned}
  \right.
 \end{equation*} 
is a John curve with parameter $\la$.

 \step{Analysis of Case A.} We claim that there exists $\mu>0$ such that the
curve 
\begin{equation*}
  \gamma_\z(s)= 
  \left\{
\begin{aligned}
& \Bigl( \z_s,  \phi(\z) +\beta + \eta \int_0^{s }\abs{\z_\r}^{2m}d\r
\Bigr),\quad\text{if
$ s\leq \bar{\bar s} = \e_0 \frac{|ZF(\z)|}{\abs{\z}^{2m}} $}
\\& \Bigl( \z_{\bar{\bar s}},  \phi(\z) +\beta + \eta \int_0^{\bar{\bar s}}
\abs{\z_\r}^{2m}d\r
+s-\bar{\bar s}  \Bigr),\quad\text{if
$ s\geq \bar{\bar s} $}
\end{aligned}
  \right.
 \end{equation*} 
is a John curve with parameter $\la$. To prove this claim, it suffices to
show 
that $- XF(\z) >\frac 14 |ZF(\z)|  $ if  Case A holds and  $\mu$ is small
enough.

From \eqref{terzina} it  follows that  $|\nabla Z F(z)| \leq
C|z|^{2m}$ for all $z\in \R^2$ and thus 
the function  $z\mapsto
 {|ZF(z)|} / {|z|^{2m}}$ is globally
Lipschitz continuous on $\R^2$. 
Let $L$ be the  Lipschitz constant. By \eqref{massimino} and by  the Case A with
sufficiently small $\mu$, we have 
\begin{equation*}
\begin{aligned}
 - \frac{XF(\z)}{|\z|^{2m}}& \geq  -\frac{XF(z)}{|z|^{2m}}-L|\z-z|
 \geq \frac 12 \frac{|ZF (z)|}{|z|^{2m}} - L d(p,q)
 \geq \frac 12 \frac{|ZF (z)|}{|z|^{2m}} - L\mu  \frac{|ZF (z)|}{|z|^{2m}}
 \\&
 \geq \frac 14 \frac{|ZF (z)|}{|z|^{2m}} 
 \geq\frac 14\frac{|ZF (\z)|}{\abs{\z}^{2m}}.
\end{aligned}
\end{equation*} 
Also the mapping $z\mapsto \bar s (z)$ in \eqref{tempogiusto}  is Lipschitz
continuous. Then, for $\mu$  small enough, in the  Case A the times $\bar s=
\bar s(z)$ and $\bar{\bar s} = \bar s(\z)$ satisfy
\begin{equation}\label{percorrenza} 
 \frac 12 \bar s  \leq \bar {\bar s}  \leq \frac 32 \bar s.
\end{equation}
Finally,  we also have
 $\abs{\z-z}\leq  d(p,q)\leq  \mu \frac{|ZF(z)|}{\abs{z}^{2m}}\leq C \mu
\abs{z}\leq\frac 12 \abs{z}$, for $\mu$ sufficiently   small.

 We are now ready to define the curve joining $p$ and $q$ and satisfying~\eqref{penna},~\eqref{sailor}.
For a suitable $H>0$, let 
 \begin{equation}\label{essesegnato} 
\wh  s=  H d(p,q).
 \end{equation}   
Then, the curve  $\gamma$ is the concatenation  of  $\g_z\big|_{[0,\wh s]}$, a
length-minimizing path $\wh\gamma$ joining  $\wh\gamma(0)=\g_z(\wh s)$ and
$\wh\gamma(1)=\gamma_\z(\wh s)$, and the  opposite  of $\g_\z\big|_{[0,\wh s]}$.

We claim that 
there exist $H>0$ and $\mu>0$ such that the curve $\gamma$ 
satisfies~\eqref{penna} and~\eqref{sailor}.

We preliminarily show that:
\begin{enumerate}[noitemsep,label={(\roman*)} ]
\item \label{itemuno}  $\wh s \leq \min\{ \bar s, \bar{\bar s}  \}$,  i.e.,  the
points $\gamma_z(\wh s)$ and $\gamma_\z(\wh s)$ belong to the first piece
of the curves $\gamma_z$ and $\gamma_\z$, respectively;
\item \label{itemdue} $d(\gamma_\z(\wh
s), \gamma_z(\wh s))\leq \frac{\lambda}{2}  \wh s$, where  $\lambda $  is the
John constant of $\gamma_\z$ and $\gamma_z$;   
\item \label{itemtre}  $\operatorname{diam}(\gamma)\leq C
d(p,q)$.
\end{enumerate}

Condition \eqref{penna} is  \ref{itemtre}. We show that
 \ref{itemuno}--\ref{itemtre} imply \eqref{sailor}.
For   $s\leq \wh s$,  by $\Delta_s \leq
\diam(\gamma_z|_{[0,s]})$ and by
the cone property \eqref{jolly} we have
\begin{equation*}
\begin{aligned}
B(\gamma(s) ,\lambda \Delta_s ) \subset B(\gamma_z(s ),\lambda  
\diam(\gamma_z|_{[0,s]}))\subset\mathrm{epi}(\phi) .
\end{aligned}
\end{equation*}
Then \eqref{sailor} holds with  $\e  = \lambda$. 
The same happens for points $\gamma_\z(s)$ with $s\leq \wh s$. 
Finally, for a point   $\wh\gamma(s^*)$ in the intermediate part, by
\ref{itemdue} we have
\begin{equation}\label{intermedia} 
\begin{aligned}
 \dist  (\wh\gamma(s^*),\mathrm{gr}(\phi) )&\geq \dist(\gamma_z(\wh
s),\mathrm{gr}(\phi) )-\frac{\lambda}{2}\wh s
 \\&
   \geq \lambda \diam(\gamma_z|_{[0,\wh s]})-\frac{\lambda}{2}\diam 
(\gamma_z|_{[0,\wh s]})
   =\frac{\lambda}{2}\diam  (\gamma_z|_{[0,\wh s]}).
\end{aligned}
\end{equation}
 In order to get a lower bound for the last diameter, we use the
length-minimizing property of $\wh \gamma$   and  property
\ref{itemdue}, which give 
\begin{equation*}
\operatorname{diam}(\wh \gamma_{[0, s^*]})\leq d(\wh \gamma(0),\wh\gamma(1))\leq \frac{\lambda}{2}\wh s
\leq    \frac{\lambda}{2}\diam\big( \gamma_z|_{[0,\wh s]}\Big).
\end{equation*}
 Therefore, we have  $\diam\big(\gamma_z|_{[0,\wh s]} + \wh\gamma|_{[0,  s^*]} \big)\leq
2 \diam\big(\gamma_z|_{[0,\wh s]})$
and then it is easy to conclude that  \eqref{sailor}
holds with $\e=  \frac{\lambda}{4} $.

Now we prove \ref{itemuno}. By \eqref{percorrenza} this is implied by 
  $\wh s\leq \frac 12 \bar s  $. 
  By \eqref{essesegnato},
  Case A, \eqref{massimino}, we have 
$
 \wh s\leq H\mu \frac{|ZF(z)|}{|z|^{2m}}=H\mu  \frac{\bar
s}{\e_0}.
$
Thus, we deduce that  \ref{itemuno} holds provided that
  \begin{equation}
  \label{123}
H\mu \leq \frac 12 \e_0.  
  \end{equation} 
  This is the first requirement on $H$ and $\mu$.
This restriction is  
compatible with  further conditions made below.
  
Next we prove  \ref{itemdue}. 
Theorem \ref{eqq} gives
\begin{equation}\label{duedidue} 
\begin{aligned}
 d(\gamma_z(\wh s), \gamma_\z(\wh s))  &
 \leq C_0\abs{\z- z}
    +C_0\min\Big\{ \frac{\abs{\Theta}^{1/2}}{\abs{z}^m }, 
 \abs{\Theta}^{\frac{1}{2m+2}}\Big\}
 \end{aligned}
\end{equation}
  where
\begin{equation*}
\begin{aligned}
 \Theta & = \phi(\z)-\phi( z)+\b -b+\int_0^{\wh s} (\eta|\z_\r |^{2m}  -
 y |z_\r |^{2m} ) d\r 
 +|z_{\wh   s}|^{2m}  \omega(z_{\wh s}, \z_{\wh s}) . 
\end{aligned}
\end{equation*}
 Let  $\Theta= \Theta_1 +\Theta_2 +\Theta_3$, with
\[
\begin{split}
 \Theta _1  & =  \phi(\z)-\phi(z)+ \b-b + |z |^{2m} \omega(z,\z) ,
\\
\Theta_2 &=  |z_{\wh s}|^{2m} \omega(z_{\wh s},\z_{\wh s})  
  -|z|^{2m} \omega(z,\z),
\\
\Theta_3 & =  \int_0^{\wh  s} ( \eta|\z_\r|^{2m}  - y|z_\r
|^{2m} )d\r .
\end{split}
 \]

The first term in the right-hand side of \eqref{duedidue}   can be estimated as
follows 
\[
  C_0\abs{\z-z}\leq C_0 d(p,q) \leq  \frac{\lambda}{8} H d(p,q) =\frac{\lambda}{8} \wh s,
\]
as soon as $H$ is large enough to ensure that  
 \begin{equation}
  \label{zeroreq}
  C_0\leq \frac{\lambda}{8} H,
 \end{equation} 
where $C_0$ is the absolute constant in \eqref{duedidue}. We used definition \eqref{essesegnato}
of $\wh s$. 

Concerning the second term in the right-hand side of \eqref{duedidue}, 
we claim that for all $j=1,2,3$ 
we have  
 \begin{equation}\label{rich} 
C_0\min\Bigl\{ |\Theta _j|^{\frac{1}{2m+2}} , \frac{|\Theta _j |^{1/2}}{|z|^{m}}
\Bigr\}\leq \frac{\lambda}{8} H d(p,q).
 \end{equation}

\noindent 
By Theorem \ref{eqq}, we have  
\[
C_0\min\{|\Theta _1 |^{\frac{1}{2m+2}} ,|\Theta _1 |^{1/2}/|z|^m
\}\leq  C_0 \delta(p,q)\leq C C_0  d(p,q) \leq 
\frac{ \lambda}{8} H d(p,q),
\]
as soon as $H$ is large enough so  that
\begin{equation}
 \label{secondoreq}
C C_0\leq \frac{\lambda}{8} H.
\end{equation} 

To evaluate  the term with $\Theta_2$, we apply the inequalities
\begin{equation*}
\begin{aligned}
|\Theta _2|&  =\Big| \Big( |z_{\wh s}|^{2m} -|z|^{2m}\Big)
\omega(z,\z)  + |z_{\wh s }|^{2m} \wh s(\eta-y)\Big|\leq
C|z|^{2m}|\z-z|\wh s
\\&
\leq C|z|^{2m}d(p,q)\wh s 
=C|z|^{2m} H d(p,q)^2.
\end{aligned}
\end{equation*}
Thus, we deduce that for some absolute constant $C_2>0$ we have
\[ 
 \frac{|\Theta _2|^{1/2}}{|z|^m} \leq  C_2  d(p,q)\sqrt{H} \leq 
C_0^{-1} \frac{\lambda}{8} H d(p,q)
\]
as soon as \begin{equation}
 \label{terzoreq}
 C_2\leq C_0^{-1}\frac{\lambda}{8}\sqrt{H}.
\end{equation} 

Finally, we estimate $\Theta_3$: 
\begin{equation*}
\begin{aligned}
 |\Theta _3| &\leq|y-\eta| \int_0^{\wh s}|z_\r |^{2m}d\r 
+|\eta|  \Big| 
 \int_0^{\wh  s} ( |z_\r |^{2m} - |\z_\r |^{2m}) d\r\Big|
 \\&
    \leq C \wh  s  |z|^{2m}\d(p,q ) +C\wh  s 
|\eta| |z|^{2m-1}\abs{z-\z} 
\\
&
    \leq C \wh  s|z|^{2m} d(p,q) =C |z|^{2m} d(p,q)^2 H,
\end{aligned}
\end{equation*}
and we end up with again with the requirement \eqref{terzoreq}.

To conclude the proof, we choose $H>0$ large enough so that 
\eqref{zeroreq}, \eqref{secondoreq} and \eqref{terzoreq} hold. This implies
\ref{itemdue}. Then we choose   $\mu>0 $ such that  
\eqref{123} holds. This implies  \ref{itemuno}.
The diameter estimate in \ref{itemtre} holds in terms of such
constants and the proof of Case A is concluded.

 \medskip

 \step{Analysis of Case B. } Let us consider the second piece of the curve from
$(z, \phi(z)+b)$,
 \begin{equation*}
  \gamma_z(s)=\Big( z_{\ol s} , t+ b+y\int_0^{\ol s}\abs{z_\r}^{2m }d\r + s-\ol s\Big)\quad\text{for $s\geq \ol s=\e_0\frac{|ZF(z)|}{|z|^{2m}}$.}
 \end{equation*}
Let also   $(\z,\phi(\z)+ \b)$  be such that Case B holds. Then, there is  a unit vector $w=(u,v)\in\R^2$ 
such that the curve $\gamma_\z$ is a
$\lambda$-John curve in $\mathrm{epi}(\phi)$ starting from $(\z, \phi(\z)+\beta) $.
When $s\geq \bar{\bar s} = \e_0 |ZF(\z)| / |\z|^{2m}$, the curve is
\[
\gamma_\z(s)=
 \Bigl( \z+\bar {\bar s}  w , \phi(\z) +\beta+\omega(w,\z)  \int_0^{\bar {\bar
s}} \abs{\z+\rho w}^{2m} d\r + s-  \bar{\bar s} \Bigr).
 \]
Note that the numbers $\bar s$ and $\bar{\bar s}$ could both vanish.
Furthermore, we will  assume without loss of generality that 
$ \operatorname{diam}\big(\gamma_z\big|_{[0,  \wh s_z ]}\big)
 \geq   \operatorname{diam}\big(\gamma_\z\big|_{[0, \wh s_\z ]}\big)$.

For $\tau\geq 0$   consider the points
$\gamma_z(\wh {s_z} )$ and $\g_\z(\wh s_\z)$, where  
\begin{equation*}
\wh s_z =\bar s +\t^{2m+2},\qquad \wh s_\z = \bar{\bar s} +\t^{2m+2}.       
                                                           \end{equation*} 

We claim that there exists  $M>0$ such that for all $p=(z,\phi(z)+b)$ and $ q
=(\z,\phi(\z)+\beta)$ for which Case B holds, if  $\t\geq 0$ satisfies 
\begin{equation}\label{giove} 
  \operatorname{diam}\Big(\gamma_z\big|_{[0,  \wh s_z ]}\Big)=\max\Bigl\{
 \operatorname{diam}\Big(\gamma_z\big|_{[0,  \wh s_z ]}\Big)
 ,  \operatorname{diam}\Big(\gamma_\z\big|_{[0, \wh s_\z ]}\Big)
 \Bigr\}=  M d(p,q),
\end{equation} 
then we have    
\begin{equation}\label{conogelato} 
 \begin{aligned}
d\Big(\g_z(\wh s_z ), & \g_\z(\wh s_\z ) \Big)
\leq  \frac{\la}{2}
 \operatorname{diam}\Big(\gamma_z\Big|_{[0, \wh s_z ]}\Big),
 \end{aligned}
\end{equation} 
 where   $\la$  is the John constant of the curves.

Notice that for any $M, p,q $ there is always a $\t$ such that
\eqref{giove} holds because  the left-hand side
of
\eqref{giove} is increasing in $\t$ and tends to $+\infty$, as $\t\to+\infty$.

We prove the claim. By the  invariance  of the distance with respect to vertical
translations we have 
\begin{equation*}
\begin{aligned}
d\Big(\g_z(\bar s +\t^{2m+2}),& \g_\z(\bar {\bar s}+\t^{2m+2}) \Big)
=
d \big(\g_z(\bar s ), \g_\z(\bar s ) \big)
 \\&\leq 
 d(\g_z(\bar s ), \g_z(0))+ d (\g_z(0),\g_\z(0))+d(\g_\z(\bar{\bar
s}),\g_\z(0) )
 \\&
 \leq \e_0\frac{|ZF(z)|}{\abs{z}^{2m}}+d(p,q)+\e_0
\frac{|ZF(\z)| }{\abs{\z}^{2m}}
 \\&
 \leq \frac{2 \e_0}{\mu}d(p,q)+d(p,q)
 \\&
= \frac{1 }{M}\Big(\frac{2 \e_0}{\mu}+1\Big)
 \max\{\diam(\g_z|_{[0, \wh s_z]}), \diam(\g_\z|_{[0, \wh s_\z]})\}, 
\end{aligned}
\end{equation*}
by \eqref{giove}. Thus \eqref{conogelato} holds  if $M$ is large enough,  and
the claim is proved. 

To conclude the proof,  we show that  the path $\gamma$, given by the
concatenation of    $\gamma_\z\big|_{[0, \wh s_\z ]}$, a length
minimizing path $\wh \gamma $ connecting $\wh \gamma(0)=
\g_z(\wh s_z )$ and $\wh\gamma(1)=\g_\z(\wh s_\z )$ and the
reverse 
of  $\gamma_z\big|_{[0, \wh s_z ]}$ satisfies   
the
$(\e, \d)$-condition.
Since the diameter estimate \eqref{penna} is contained in the claim above, we are left with the proof of~\eqref{sailor}.

Let $q$ be a point of $\gamma$.  If 
$q=\gamma_z(s)$ with $s\leq \wh s_z$ or $q=\gamma_\z(s)$ with $s\leq \wh s_\z$,
then
\eqref{sailor} follows with $\e=\lambda$ from the John property \eqref{jolly}. 
If   
$q=\wh\gamma (s^*)$, then  we argue as in \eqref{intermedia}. precisely
\begin{equation}\label{jelafamo} 
\begin{aligned}
\dist (\wh\gamma(s^*), \operatorname{gr}\phi   ) &\geq 
  \dist (\gamma_z( \wh s_z),   \operatorname{gr}\phi   ) -
d(\wh\gamma(s^*),\gamma_z( \wh s_z)) 
\\& \geq \lambda \operatorname{diam}(\gamma_z|_{[0, \wh s_z]}) -\diam(
\wh \gamma)\geq \frac{\lambda}{2} \operatorname{diam}(\gamma_z|_{[0, \wh s_z]}),
\end{aligned}
\end{equation}
by \eqref{conogelato}. 
Finally, to get a lower estimate of the latter diameter with $\diam (\gamma_z|_{[0, \wh s_z]} + 
\hat \gamma|_{[0, s^*]})$, which will 
give the John property, it suffices to use the length minimizing property of $\wh\gamma $ 
\begin{equation*}
\begin{aligned}
  \diam(\wh\gamma|_{[0, s^*]}) &\leq  d(\wh \gamma(0), \wh \gamma(1)) \leq  
   \frac{\lambda}{2} \operatorname{diam}\big(\gamma_z\big|_{[0, \wh s_z ]}\big).
\end{aligned}
\end{equation*}
Thus, as in Case A, we get the correct lower bound  for the last line of \eqref{jelafamo} and the proof is easily concluded. 
\end{proof}

%

 \section{Bounded admissible domains are uniform}\label{ultimissima} 
In this section we prove Theorems \ref{principale} and   \ref{Ahl} in the
case of a \emph{bounded} $m$-admissible domain.
Now we assume that $m\in\N$ is an integer.

\begin{proof}[Proof of Theorem \ref{principale}]
 Let $\Omega\subset\R^3$ be an $m$-admissible domain. 
By a 
 standard localization argument (see
e.g.~\cite[Proposition~2.5]{MontiMorbidelli05b}), it suffices 
 to show that for all $p_0\in\p\Omega$ there is a  neighborhood $
A_{p_0}$ of $p_0$ in $\R^3$ such that for all $p,q\in A_{p_0}$ there is a 
continuous curve $\gamma:[0,1]\to \Omega\cap A_{p_0}$ satisfying 
 $\gamma(0)= p $ and $\gamma(1)=q$ and such that  \eqref{penna} and
\eqref{sailor} hold.
 
There are two cases:  
 \begin{enumerate}[noitemsep]
 \item \label{bagno}  $p_0$ is a noncharacteristic point, i.e., $\Span\{X(p_0),
Y(p_0)\}$ is not contained in $T_{p_0}\partial\Omega$.
  
 \item $p_0$ \label{cucina} is a characteristic point of $\partial \Omega$. 

 \end{enumerate}

 In Case \ref{bagno}, the claim is proved in 
\cite[Theorem~1.1]{MontiMorbidelli05b}.  
To use this result,  we need a $C^\infty$ boundary and smooth vector fields.
For this reason
we require  $m\in\N$.  

In the Case  \ref{cucina}, in a neighborhood of $p_0$ the 
 boundary of $\Omega$ is a graph of the type $t =
\phi(z)$ for an $m$-admissible function $\phi \in C^\infty(D)$ for some open
set $D\subset \R^2$.
The claim is proved in Sections \ref{CONO} and 
\ref{disconosco}. 
\end{proof}

 \begin{proof}[Proof of Theorem \ref{Ahl}] 
By compactness, we can cover $\p\Omega$ with a finite union of $m$-admissible
graphs, together with a compact subset $K\subset\p\Omega$ containing only
noncharacteristic points. 

At points $p\in K$, the  Ahlfors estimates \eqref{Renzi} is proved
in
\cite[Corollary 1]{MontiMorbidelli02}. To use this result, we need a smooth
boundary and smooth vector fields ($m\in\N$).

On $m$-admissible graphs, the Ahlfors estimate is proved in 
Section~\ref{DiMaio}.  
\end{proof}

 \footnotesize

 \def\cprime{$'$} \def\cprime{$'$}
\providecommand{\bysame}{\leavevmode\hbox to3em{\hrulefill}\thinspace}
\providecommand{\MR}{\relax\ifhmode\unskip\space\fi MR }
\providecommand{\MRhref}[2]{%
  \href{http://www.ams.org/mathscinet-getitem?mr=#1}{#2}
}
\providecommand{\href}[2]{#2}


\normalsize

\end{document}